\documentclass[a4paper,10pt]{elsarticle}

\usepackage[utf8]{inputenc}
\usepackage{amsmath}
\usepackage{amsfonts}
\usepackage{amssymb}
\usepackage{amsthm}
\usepackage{relsize}
\usepackage{caption}
\usepackage{multicol}
\usepackage{multirow}
\usepackage[english]{babel}
\usepackage{fontenc}
\usepackage{graphicx}
\usepackage{todonotes}
\usepackage{color}
\usepackage{enumerate}
\usepackage{fixmath}
\usepackage{bm}
\usepackage{siunitx}
\usepackage{hyperref}

\newtheorem{remark}{Remark}[section]
\newtheorem{prop}{Proposition}[section]

\newtheorem{lemma}{Lemma}[section]

\newcommand{\Proof}{\textbf{Proof.}~}
\newcommand{\cvd}{\begin{flushright}$\square$\end{flushright}}

\newcommand{\V}{\textnormal{$\mathbf{v}$}}
\newcommand{\N}{\textnormal{$\mathbf{n}$}}

\newcommand{\X}{\textnormal{$\mathbf{x}$}}

\newcommand{\dO}{\,{\rm d}\Omega}

\newcommand{\dPP}{\,{\rm d}\PP}
\newcommand{\dpPP}{\,{\rm d}\partial\PP}

\newcommand{\de}{\,{\rm d}e}
\newcommand{\df}{\,{\rm d}f}

\newcommand{\PP}{{\text P}}

\newcommand\Om{{\Omega}}

\newcommand\cc{\boldsymbol c}

\newcommand\bPi{\boldsymbol \Pi}

\newcommand{\VFN}{V^\nabla_h(f)}
\newcommand{\VFNT}{\widetilde{V}^\nabla_h(f)}
\newcommand{\VF}{V_h^\Delta(f)}
\newcommand{\VFT}{\widetilde{V}^\Delta_h(f)}

\newcommand{\VFNGen}{V^\nabla_{h,k}(f)}
\newcommand{\VFNTGen}{\widetilde{V}^\nabla_{h,k}(f)}
\newcommand{\VFGen}{V_{h,k}^\Delta(f)}
\newcommand{\VFTGen}{\widetilde{V}^\Delta_{h,k}(f)}

\newcommand{\DZ}{\textbf{D0}}
\newcommand{\DU}{\textbf{D1}}

\newcommand{\CZ}{\mathcal{C}0}
\newcommand{\CU}{\mathcal{C}1}

\newcount\colveccount
\newcommand*\colvec[1]{
        \global\colveccount#1
        \begin{pmatrix}
        \colvecnext
}
\def\colvecnext#1{
        #1
        \global\advance\colveccount-1
        \ifnum\colveccount>0
                \\
                \expandafter\colvecnext
        \else
                \end{pmatrix}
        \fi
}

\newcommand{\coor}[1]{{#1}}

\begin{document}

\begin{frontmatter}

\title{A $C^1$ Virtual Element Method \\ on polyhedral meshes}

\author[add1,add2]{L. Beir\~{a}o da Veiga}
\ead{lourenco.beirao@unimib.it}
\author[add1]{F. Dassi}
\ead{franco.dassi@unimib.it}
\author[add1,add2]{A. Russo}
\ead{alessandro.russo@unimib.it}

\address[add1]{Department of Mathematics and Applications, University of Milano - Bicocca,\\ 
Via Cozzi 53, I-20153, Milano (Italy)}
\address[add2]{IMATI-CNR, 27100 Pavia (Italy)}


\begin{abstract}
The purpose of the present paper is to develop $C^1$ Virtual Elements in three dimensions for linear elliptic fourth order problems, motivated by the difficulties that standard conforming Finite Elements encounter in this framework. We focus the presentation on the lowest order case, the generalization to higher orders being briefly provided in the Appendix. The degrees of freedom of the proposed scheme are only 4 per mesh vertex, representing function values and gradient values. Interpolation error estimates for the proposed space are provided, together with a set of numerical tests to validate the method at the practical level.
\end{abstract}


\end{frontmatter}


\noindent 
\textsl{The copyright of the original paper is owned by Elsevier. 
The original publication (DOI 10.1016/j.camwa.2019.06.019) appears on the journal Computer \& Mathematics with Applications and can be found at the journal website}

\begin{center}
\url{https://www.journals.elsevier.com/computers-and-mathematics-with-applications}.
\end{center}

\section{Introduction}

Fourth order partial differential equations are used to describe many different physical phenomena such as plate bending problems and evolution of transition interfaces.
In standard $H^2$ conforming finite elements \coor{these} problems require a globally $C^1$ piecewise polynomial space and, to get such regularity on a general unstructured partition, a very high minimal polynomial degree is needed.
In~\cite{MR2983014} there is an analysis on the minimal degree required to build a finite element space in the Sobolev space $H^m(\mathbb{R}^d)$ via the Finite Element Method.
In particular, the authors show that for $H^2$ the minimal polynomial degree is 5 in two dimensions and 9 in three dimensions.
It is easy to understand that such compulsory high polynomial degree increases the computational effort and makes the method unpractical in many situations.
For instance, a conforming $C^1$ finite element space on a tetrahedral mesh will require 220 degrees of freedom per element~\cite{MR0350260}.
To avoid this high computational effort, there are possible alternatives in the literature, such as non-conforming and discontinuous 
schemes (see for instance \cite{morley1968,MR2142191,MR1915664,MR2295480}), making use of a mixed formulation (see for instance \cite{MixedBrezzi,Falk1978,MR2453216,MR899701}) or construct more complex discrete spaces obtained by some macro-element strategy (see for instance \cite{MR2817542,zhang2008,Lai2004}).
It must be mentioned that $C^1$ finite elements are also important because they can be used to build exact discrete Stokes complexes, see for instance \cite{N1,N2} and the citations thereof. 

Another strategy to get a conforming discrete approximation space in $H^2$ is to follow the recently born Virtual Element Method (VEM). The VEM is a novel generalization of the finite element method, introduced in \cite{volley, autostoppisti}, that allows to use general polygonal/polyhedral meshes and which has been already succesfully applied to a large number of problems 
(a very brief list being 
\cite{antonietti2014stream,
benedetto2014virtual,
gain2014virtual,
gain2015topology,
mora2015virtual,
da2015virtual,
benedetto2016hybrid,
cangiani2016nonconforming,
genCoeff,
wriggers2016virtual,
bertoluzza2017bddc,
Brenner,
caceres2017mixed,
Cangiani:apos,
vacca2018h,
fumagalli2018dual,
da2018virtual,
dassi2018exploring,
nguyen2018virtual,
da2018lowest,
ARTIOLI2018978,
artioli2018asymptotic,
Brenner2}).
The Virtual Element Method is not restricted to piecewise polynomials but avoids nevertheless the explicit integration of non-polynomial shape functions by a wise choice of the degrees of freedom and an innovative construction of the stiffness matrix. Although the main motivation of VEM is the use of general polytopal partitions, its flexibility can lead also to different advantages. One, initiated in \cite{Brezzi:Marini:plates,BM13} and further investigated in~\cite{ABSVu,Paper_con_Gonzalo,Mora1,Mora2}, is the possibility to develop $C^1$ conforming spaces, still keeping the accuracy order and the number of degrees of freedom at a reasonable level.
More specifically, the lowest degree requires only three degrees of freedom for each vertex
independently for the shape of the elements.

Since all the papers above are limited to the two-dimensional case, the purpose of the present contribution is to develop $C^1$ Virtual Elements in three dimensions. We focus the presentation on the lowest order case (the generalization to higher orders being briefly provided in the Appendix) for the sake of exposition but also since we believe this is the most interesting choice in practice. Developing a discrete Virtual Element space in three dimensions needs first the construction of ad-hoc two dimensional spaces on the faces (polygons) of the polyhedra, one for the function values and one for the normal derivatives. The final degrees of freedom of the proposed scheme are simply 4 per mesh vertex, representing function values and gradient component values. Consequently, although VEM can be applied to general polyhedral meshes, the proposed method becomes appealing also for standard tetrahedral meshes. After developing the method and the associated degrees of freedom, we prove interpolation estimates for the provided discrete space in standard $L^2,H^1$ and $H^2$ Sobolev norms. Finally, we show a set of numerical tests on classical linear fourth order elliptic problems that validate the method at the practical level. We also include a comparison, for the standard Poisson problem, with the $C^0$ VEM in 3D. It is finally worth noticing that, although the present paper is focused on the linear elliptic case, the presented discrete space could be used also to discretize more complex nonlinear problems, such as the Cahn-Hilliard equation.

The paper is organized as follows. In Section~\ref{sec:prob} we outline the range of fourth order linear problems under consideration.
In Section~\ref{sec:spaces} we describe the $C^1$ virtual element spaces and provide a set of associated degrees of freedom. In Section~\ref{sec:virtForms} we present the numerical method, that is the VE discretization of the problem. In Section~\ref{sec:interpo}, we prove the interpolation and convergence estimates.
In Section~\ref{sec:numExe}, we present the numerical results.
Finally, in~\ref{appendix:general} we briefly outline the extension to the higher order case.


\section{Continuous Problem}\label{sec:prob}

Let $\Omega\subset\mathbb{R}^3$ be a bounded domain,
we consider the following problem: find $u(\X):\Omega\to \mathbb{R}$ such that 
\begin{equation}
\left\{
\begin{array}{rlll}
c_1\Delta^2 u - c_2\,\Delta u + c_3\,u  &=&\, f\phantom{_1}\quad\quad&\textnormal{in }\Omega\\[0.2em]
u &=&\,g_1\quad\quad&\textnormal{on }\partial\Omega\\
\partial_n u &=&\,g_2\quad\quad&\textnormal{on }\partial\Omega
\end{array}
\right.,
\label{eqn:probGenStrong}
\end{equation}
where $c_1>0$, $c_2,c_3\geq 0$ are constant coefficients, 
$f\in L^2(\Omega)$ is the forcing term, 
$\partial_n u$ is the partial derivative of $u$ with respect to the boundary normal $\N$,
$g_1\in H^{3/2}(\partial\Omega)$ and $g_2\in H^{1/2}(\partial\Omega)$ are the Dirichlet data.
Note that we have considered Dirichlet boundary conditions only for simplicity of exposition,
the extension to more general cases is trivial.
To define the variational formulation of Problem~\eqref{eqn:probGenStrong},
we introduce the bilinear forms
\begin{equation}
\begin{array}{rll}
a^\Delta(v,\,w) &:= \mathlarger{\int}_\Omega \nabla^2 v\,:\,\nabla^2 w\,\dO  &\quad\forall v,\,w\in H^2(\Omega)\,,\\[1em]
a^\nabla(v,\,w) &:= \mathlarger{\int}_\Omega \nabla v\,\cdot\,\nabla w\,\dO  &\quad\forall v,\,w\in H^1(\Omega)\,, \\[1em]
a^0(v,\,w) &:= \mathlarger{\int}_\Omega v\,w\,\dO  &\quad\forall v,\,w\in L^2(\Omega)\,,
\end{array}
\label{eqn:biliForms}
\end{equation}
where all the previous symbols refers to the standard notation for functional spaces.
We define
\begin{equation*}
V(\Omega) := \big\{v\in H^2(\Omega)\,:\,\,v = g_1\quad\text{and}\quad\partial_n u = g_2\quad\textnormal{on }\partial\Omega\big\}\,,
\end{equation*}
and
\begin{equation*}
V_0(\Omega) := \big\{v\in H^2(\Omega)\,:\,\,v = 0\quad\text{and}\quad\partial_n u = 0\quad\textnormal{on }\partial\Omega\big\}\,.
\end{equation*}
The weak formulation of Problem~\eqref{eqn:probGenStrong} reads: find $u\in V(\Omega)$ such that
\begin{equation}
c_1\,a^\Delta(u,\,v) + c_2\,a^\nabla(u,\,v) + c_3\,a^0(u,\,v) = (f,\,v)_\Om\quad\quad\quad\forall v\in V_0(\Omega)\,,
\label{eqn:probGenWeak}
\end{equation}
where $(\cdot,\cdot)_\Om$ is the standard $L^2$-inner product.
\coor{Due to the coercivity of $a^\Delta(\cdot,\cdot)$ on the space $V_0(\Omega)$ the Lax-Milgram lemma yields the well posedness of the above problem.}

\begin{remark}
The method proposed in this paper can be applied also to second order elliptic problems (that is $c_1=0$ and $c_2>0$)
to get a $C^1$ conforming solution, as shown in the numerical test \coor{in Section}~\ref{sub:c0Comp}.
\end{remark}

\begin{remark}
The method of the present paper can be extended to the variable coefficient case by combining this construction with the approach in~\cite{genCoeff,chinaSere}.
\end{remark}


\section{$C^1$ Virtual Element Spaces}\label{sec:spaces}

Let $\Omega_h$ be a discretization of $\Omega$ composed by polyhedrons.
As in the standard VEM framework, we define the local space and projection operators in a generic polyhedron $P$
and then we glue such local virtual element spaces to define the discrete global space, $V_h(\Omega_h)$.

We achieve this goal in two steps. 
We first define virtual spaces on faces, \coor{Sections}~\ref{sub:VSOnFace1} and~\ref{sub:VSOnFace2},
then we define virtual spaces on polyhedrons \coor{in \coor{Section}~\ref{sub:VSOnPoly}}.
Since the virtual face spaces essentially correspond to 2D virtual spaces already defined in~\cite{volley,BM13,cinesi_plates},
we only make a brief review and refer to such papers for a deeper description.

In order to derive the convergence theory, we will need the following assumptions on the mesh $\Omega_h$:
\begin{itemize}
 \item[(\textbf{A1})] Each element $P$ is star shaped with respect to a ball $B_P$ 
 whose radius is uniformly comparable with the polyhedron diameter, $h_P$.
 \item[(\textbf{A2})] Each face $f$ is star shaped with respect to a disc $B_f$
 whose radius is uniformly comparable with the face diameter, $h_f$.
 \item[(\textbf{A3})] Given a polyhedron $P$ all its edge lengths and face diameters are uniformly comparable with respect to its diameter $h_P$.
\end{itemize}

\begin{remark}
It is easy to check that under assumptions (A1), (A2) and (A3) each polyhedron is the union of uniformly 
shape-regular tetrahedrons all sharing a central vertex.
\label{L:tets}
\end{remark}

Let $\mathcal{D}\subset\mathbb{R}^d$, 
from now we refer to the polynomial space in $d$-variables of degree 
lower or equal to $k$ as $\mathbb{P}_k(\mathcal{D})$.


\subsection{Virtual element nodal space $\VFN$}\label{sub:VSOnFace1}

We define the preliminary space on each face $f\in\partial P$
\begin{eqnarray*}
\VFNT := \bigg\{v_h\in H^1(f)\:&:&\Delta_\tau\,v_h\in\mathbb{P}_0(f)\,, \\[-1em]
\phantom{\VFNT := \bigg\{v_h\in H^1(f)\:}&\phantom{:}&
v_h|_{\partial f}\in C^0(\partial f)\,,\:v_h|_e\in\mathbb{P}_1(e)\:\:\forall e\in\partial f\bigg\}\,. 
\end{eqnarray*}
where $\Delta_\tau$ is the Laplace operator in the local face variables.

We consider the standard VEM setting proposed in~\cite{volley} and 
we build the projection operator $\Pi^\nabla_f:\VFNT\to\mathbb{P}_1(f)$, 
defined by 
\begin{equation}
\left\{
\begin{array}{rll}
a_f^\nabla(\Pi^\nabla_f v_h\,, p_1) &=a_f^\nabla(v_h\,, p_1) &\forall p_1\in\mathbb{P}_1(f)\\[0.4em]
(\Pi^\nabla_f v_h,\,1)_{\partial f} &= (v_h,\,1)_{\partial f} \\
\end{array}\,,
\right.
\label{eqn:piNablaFace}
\end{equation}
where 
\begin{eqnarray*}
a_f^\nabla(v_h\,,w_h) &:=& \int_f \nabla_\tau v_h\,\cdot\nabla_\tau\,w_h\,\df\,,
\end{eqnarray*}
here $\nabla_\tau$ is the gradient operator in the local face coordinates and 
$(\cdot,\cdot)_{\partial f}$ is the standard $L^2$ inner product over the boundary of $f$.

The projection operator $\Pi^\nabla_f$ is well defined and 
uniquely determined by the values of the function $v_h$ at the vertices of the face $f$~\cite{volley}.

Moreover, starting from the space $\VFNT$ and the projection operator $\Pi^\nabla_f$, 
we are able to define the nodal space 
\begin{equation}
\VFN := \left\{ v_h\in\VFNT\::\:\int_f\Pi^\nabla_f v_h\df = \int_f v_h\df\right\}\,,
\label{eqn:nodalEnhanced}
\end{equation}
whose degrees of freedom are the values of $v_h$ at the vertices of $f$~\cite{volley,projectors}.

\subsection{Virtual element $C^1$ space $\VF$}\label{sub:VSOnFace2}

We start from the preliminary space
\begin{eqnarray*}
\VFT := \Bigg\{v_h\in H^2(f)\:&:&\: \Delta^2_\tau\,v_h\in\mathbb{P}_1(f),\\[-1em]
\:&\phantom{:}&\: v_h|_{\partial f}\in C^0(\partial f),\:\:v_h|_e\in\mathbb{P}_3(e)\:\:\forall e\in\partial f\,, \\[-0.5em]
\:&\phantom{:}&\: \nabla_\tau\,v_h|_{\partial f}\in [C^0(\partial f)]^2,\:\:
\partial_{n_e} v_h\in\mathbb{P}_1(e)\:\:\forall e\in\partial f\Bigg\}\,,
\end{eqnarray*}  
where $\partial_{n_e}v_h$ denotes the outward normal derivative to each edge.

We consider the projection operator $\Pi^\Delta_f:\VFT\to\mathbb{P}_2(f)$ defined by 
the following relations
\begin{equation}
\Bigg\{
\begin{array}{rll}
a_f^\Delta\left(\Pi^\Delta_f v_h\,, p_2\right) &=a_f^\Delta\left(v_h\,, p_2\right) &\forall p_2\in\mathbb{P}_2(f)\\[0.2em]
\left(\Pi^\Delta_f v_h\,, p_1\right)_{\partial f} &= \left(v_h\,, p_1\right)_{\partial f} &\forall p_1\in\mathbb{P}_1(f)\\
\end{array}\,,
\label{eqn:piDeltaFace}
\end{equation}
where  
\begin{eqnarray*}
a_f^\Delta\left(v_h\,,w_h\right) &:=& \int_f \nabla_\tau^2\,v_h\,:\nabla_\tau^2\,w_h\,\df\,,
\end{eqnarray*}
is a bilinear operator and $\nabla^2_\tau$ refers to the Hessian in the face local coordinates system.

The projection operator $\Pi^\Delta_f:\VFT\to\mathbb{P}_2(f)$ is well-defined 
and it is uniquely determined by the values of the function, $v_h(\nu)$,
and the values of the gradient, $\nabla_\tau v_h(\nu)$, at the face vertices~\cite{ABSVu,BM13,cinesi_plates}. 

We exploit the space $\VFT$ and the projection operator $\Pi^\Delta_f$ 
to define the virtual element $C^1$ space
\begin{equation}
\VF:= \left\{ v_h\in\VFT\::\:\int_f \Pi^\Delta_f v_h\,\,p_1\df = \int_f v_h\,\,p_1\df,
\quad\forall p_1\in\mathbb{P}_1(f)\right\}\,.
\label{eqn:c1Enhanced} 
\end{equation}
A set of degrees of freedom for $\VF$ is given by the function and the function gradient 
values at the face vertices~\cite{BM13,ABSVu}. 

\begin{remark}
We would like to underline that the additional properties on face integrals 
required by the spaces $\VFN$ and $\VF$, namely
\begin{equation}
\int_f v_h\df = \int_f\Pi^\nabla_f v_h\df \,,\qquad\forall v_h\in\VFN \
\label{eqn:condNodal}
\end{equation}
and 
\begin{equation}
\int_f v_h\,\,p_1\df = \int_f \Pi^\Delta_f v_h\,\,p_1\df
\qquad\forall p_1\in\mathbb{P}_1(f)\,,\qquad\forall v_h\in\VF
\label{eqn:piDeltaFaceEnhancing}
\end{equation}
will be essential to define our virtual scheme on polyhedrons.
\end{remark}

Finally, we make use of the $L^2$-projection operator on faces $\bPi_f^0:[L^2(f)]^2\to[\mathbb{P}_0(f)]^2$ 
to approximate the gradient of a generic function $v_h\in \VF$. 
Such projection operator is defined by these relations 
\begin{equation}
\int_f \bPi_f^0\,(\nabla_\tau v_h) \cdot \cc \df = \mathlarger{\int_f} \nabla_\tau v_h \cdot \cc \df\,, 
\quad\quad\forall \cc\in[\mathbb{P}_0(f)]^2\,.\\
\label{eqn:defPiFGrad}
\end{equation}

This projection operator is computable from the degrees of freedom of $\VF$.
Indeed, let us consider the right hand side of Equation~\eqref{eqn:defPiFGrad}
$$
\mathlarger{\int_f} \nabla_\tau v_h \cdot \cc\df = 
\mathlarger{\int_{\partial f}} v_h\,(\N_f\cdot \cc) \df = 
\sum_{e\in \partial f} (\N_e\cdot \cc)\,\mathlarger{\int_e}\,v_h\de\,,
$$
the last integral is \emph{exactly} computable 
since the virtual function $v_h$ is a polynomial of degree 3 on the edges and
such edge polynomials are uniquely determined by the degrees of freedom of $\VF$.


\subsection{Virtual element space in $P$}\label{sub:VSOnPoly}

Given a polyhedron $P\in\Omega_h$ we consider the preliminary space 
\begin{eqnarray}
\widetilde{V}_h(P) := \bigg\{v_h\in H^2(P)\:&:&\: \Delta^2\,v_h\in\mathbb{P}_2(P), \nonumber \\[-0.3em]
\:&\phantom{:}&\: 
v_h|_{S_P}\in C^0(S_P)\,,
\nabla v_h|_{S_P}\in[C^0(S_P)]^3\,,\nonumber \\
\:&\phantom{:}&\: 
v_h|_f\in \VF\,,\,
\partial_{n_f} v_h|_f\in \VFN
\quad\forall f\in\partial P
\bigg\}\,,\nonumber \\
\label{eqn:defVTildeP}
\end{eqnarray}
where $S_P$ denotes the skeleton (the union of all edges) of the polyhedron $P$.

This space is composed by functions whose bi-Laplacian is a polynomial of degree 2.
The restriction of such functions on each face is a two dimensional $C^1$ virtual function,
see \coor{Section}~\ref{sub:VSOnFace2},
while their normal derivative on each face is a $C^0$ virtual function,
see \coor{Section}~\ref{sub:VSOnFace1}.

To build a suitable virtual element space in a polyhedron $P$,
we define two sets of linear operators from $\widetilde{V}_h(P)$ to $\mathbb{R}$:
\begin{center}
\begin{minipage}{0.9\textwidth}
\begin{itemize}
 \item[\DZ:] the values of the function at the vertices, $v_h(\nu)$; 
 \item[\DU:] the values of the gradient components at the vertices, $\nabla v_h(\nu)$.
\end{itemize} 
\end{minipage}
\end{center}

We define the projection operator $\Pi^\Delta_P:\widetilde{V}_h(P)\to\mathbb{P}_2(P)$ by the following relations
\begin{equation}
\Bigg\{
\begin{array}{rll}
a_P^\Delta\left(\Pi^\Delta_P v_h\,, p_2\right) &=a_P^\Delta\left(v_h\,, p_2\right) &\forall p_2\in\mathbb{P}_2(P)\\[0.2em]
\left(\Pi^\Delta_P v_h\,, p_1\right)_{\partial P} &= \left(v_h\,, p_1\right)_{\partial P} &\forall p_1\in\mathbb{P}_1(P)\\
\end{array}\,,
\label{eqn:piDelta}
\end{equation}
where we introduced
\begin{eqnarray}
a_P^\Delta\left(v_h\,,w_h\right) &:=& \int_P \nabla^2\,v_h\,:\nabla^2\,w_h\,\dPP\,,\label{eqn:aDelta}
\end{eqnarray}
and $(\cdot,\cdot)_{\partial P}$ is the standar $L^2$ inner product over the boundary of $P$.

As usual in VEM, the second condition in Equation~\eqref{eqn:piDelta} is needed 
to select an element from the non-trivial kernel of the operator $a^\Delta(\cdot,\,\cdot)_P$.

\begin{lemma}
The operator $\Pi^\Delta_P:\widetilde{V}_h(P)\to\mathbb{P}_2(P)$ is computable and uniquely determined by 
the values of the linear operators $\DZ$ and $\DU$.
\label{lem:piDeltaPoly}
\end{lemma}
\Proof Let us consider the first condition in Equation~\eqref{eqn:piDelta}.
The main issue is how to compute the right hand side
since it involves the virtual function $v_h$.
We integrate by parts and we get 
\begin{eqnarray*}
a_P^\Delta\left(v_h\,, p_2\right) &=& \int_P \nabla^2\,v_h\,:\nabla^2p_2\dPP  = \\ &=&
- \int_P \nabla\,v_h\,\cdot\textbf{div}(\nabla^2p_2)\dPP + \int_{\partial P} \nabla\,v_h\cdot\big[(\nabla^2p_2)\,\N\big]\df = \\ &=&
\sum_{f\in \partial P} \int_f \nabla\,v_h\cdot\big[(\nabla^2p_2)\,\N_f\big]\df\,.
\end{eqnarray*}
Then, we make the following orthonormal vector decomposition
$$
\nabla\,v_h =  
(\nabla\,v_h\cdot\V_{f,\widetilde{x}})\,\V_{f,\widetilde{x}} + 
(\nabla\,v_h\cdot\V_{f,\widetilde{y}})\,\V_{f,\widetilde{y}} +
(\nabla\,v_h\cdot\N_f)\,\N_f\,,
$$
where $\V_{f,\widetilde{x}}$ and $\V_{f,\widetilde{y}}$ are three dimensional unit vectors (tangent to the face)
which identify the local two dimensional coordinate system of $f$, and $\N_f$ is the outward pointing normal of the face, 
see Figure~\ref{fig:faceCoord}.
We plug this decomposition in the previous equation and we get
\begin{eqnarray}
a_P^\Delta\left(v_h\,, p_2\right) &=& \sum_{f\in \partial P}\bigg[ \omega_{\V_{f,\widetilde{x}}}\int_f (\nabla\,v_h\cdot\V_{f,\widetilde{x}})\df +
\omega_{\V_{f,\widetilde{y}}}\int_f (\nabla\,v_h\cdot\V_{f,\widetilde{y}})\df + \nonumber \\
&&\hspace{0.7cm}+\,\omega_{\N_f}\int_f (\nabla\,v_h\cdot\N_f)\df \bigg]\,, 
\label{eqn:intDivision}
\end{eqnarray}
where 
$$
\omega_{\V_{f,\widetilde{x}}} := \V_{f,\widetilde{x}}\cdot\big[(\nabla^2p_2)\,\N_f\big]\,,\quad\quad
\omega_{\V_{f,\widetilde{y}}} := \V_{f,\widetilde{y}}\cdot\big[(\nabla^2p_2)\,\N_f\big]\,,
$$
and
$$
\omega_{\N_f} := \N_f\cdot\big[(\nabla^2p_2)\,\N_f\big]\,,
$$
are constant values so we can move them out from the integral over the face $f$.
Then, the integrals in Equation~\eqref{eqn:intDivision} are computable using 
the face projectors of the previous \coor{Section}s, which in turn are uniquely defined by the values of $\DZ$ and $\DU$.
More specifically, since for \eqref{eqn:defVTildeP} we have $(\nabla\,v_h\cdot\N_f)\in \VFN$, 
we can exploit the standard nodal projection $\Pi_f^\nabla$ and condition~\eqref{eqn:condNodal}.
Furthermore, since both $(\nabla\,v_h\cdot\V_{f,\widetilde{x}})$ and $(\nabla\,v_h\cdot\V_{f,\widetilde{y}})$ 
correspond to tangent derivatives of $v_h|_f$ and $v_h|_f\in\VF$ those integrals can be computed using~\eqref{eqn:defPiFGrad}.
Finally, the last condition of Equation~\eqref{eqn:piDelta} involves only integrals on faces of $P$ and 
thus it is computable from $\DZ$ and $\DU$ recalling property~\eqref{eqn:piDeltaFaceEnhancing}. 

\begin{figure}[!htb]
\begin{center}
\includegraphics[width=0.35\textwidth]{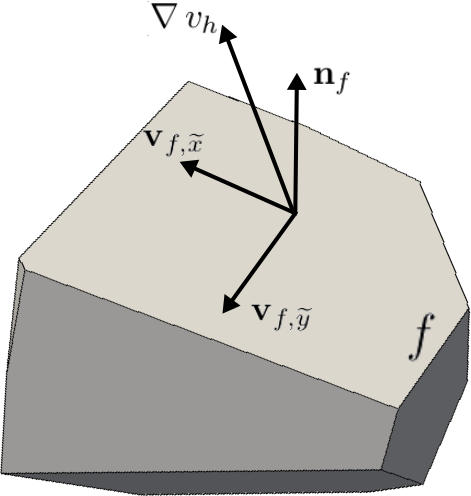}\\
\end{center}
\caption{The unit vectors $\V_{f,\widetilde{x}}$, $\V_{f,\widetilde{y}}$ and $\N_f$ for the face $f$.}
\label{fig:faceCoord}
\end{figure}

\cvd

Now we are ready to define the local virtual element space, $V_h(P)$
\begin{equation}
V_h(P) := \Bigg\{v_h\in \widetilde{V}_h(P)\::\: 
\mathlarger{\int_P} \Pi^\Delta_P v_h\,\,p_2\dPP = \mathlarger{\int_P} v_h\,\,p_2\dPP\,,\:\:\forall p_2\in\mathbb{P}_2(P)
\Bigg\}\,.
\label{eqn:VirtOnPoly}
\end{equation}

It is trivial to check that 
$$
\mathbb{P}_2(P)\subseteq V_h(P)\,.
$$

\begin{lemma}
The set of linear operators $\DZ$ and $\DU$ are a set of degrees of freedom for the space $V_h(P)$.
\end{lemma}
\Proof A function $w_h\in\widetilde{V}_h(P)$ is the solution of a well-posed 
bi-Laplacian problem defined in $P$,
whose forcing term is a polynomial of degree 2, 
\begin{equation}
\Delta^2\,w_h\in\mathbb{P}_2(P)\,, 
\end{equation}
and its Dirichlet boundary data are
\begin{equation}
w_h|_f\in \VF\,,\quad 
\partial_{n_f} w_h|_f\in \VFN
\quad \forall f\in\partial P\,.
\end{equation}
First of all, recalling the definition of the face spaces and their associated degrees of freedom,
it is easy to check that $\DZ$ and $\DU$ constitute a set of degrees of freedom 
for the boundary space $\widetilde{V}_h(P)|_{\partial P}$.

The dimension of $\widetilde{V}_h(P)$ is equal to the dimension of the data space, i.e. 
the dimension of the loading term plus the dimension of the boundary data space.
In this particular case we have that 
\begin{itemize}
 \item the dimension of the loading term is 10 since we are dealing with polynomials of degree 2 in the three dimensional space and
 \item the dimension of the boundary data space is given by the sum of the cardinality of $\DZ$ and $\DU$, $\#\{\DZ\}+\#\{\DU\}$.
\end{itemize}
We know that $V_h(P)$ is a subspace of $\widetilde{V}_h(P)$ obtained by imposing 
\begin{equation}
\int_P \Pi^\Delta_P v_h\,\,p_2\dPP = \int_P v_h\,\,p_2\dPP\quad\quad\forall p_2\in\mathbb{P}_2(P)\,, 
\label{eqn:piDeltaConst}
\end{equation}
which can be re-written as a set of 10 linear equations.
We deduce that 
$$
\dim(V_h(P))\geq \dim(\widetilde{V}_h(P)) - 10 = \#\{\DZ\}+\#\{\DU\}\,.
$$
Therefore, once we prove that a generic function $v_h\in V_h(P)$ with vanishing $\DZ$ and $\DU$ values 
is the zero element of $V_h(P)$,
we deduce that
$$
\dim(V_h(P)) = \#\{\DZ\}+\#\{\DU\}\,,
$$
and this will complete the proof.

To achieve this goal, 
suppose that $v_h\in V_h(P)$ vanishes on $\DZ$ and $\DU$.
Then
$$
v_h|_f=0,\quad\text{ and }\quad\partial_{n_f} v_h|_f=\textbf{0}\quad\quad\forall f\in\partial P\,.
$$
Moreover, since $\DZ$ and $\DU$ are zero, we have that $\Pi^\Delta v_h=0$, 
Lemma~\ref{lem:piDeltaPoly}, and therefore, recalling~\eqref{eqn:VirtOnPoly}, 
\begin{equation}
\int_P v_h\,\,p_2\dPP = \int_P \Pi^\Delta_P v_h\,\,p_2\dPP = 0\,\quad\quad\forall p_2\in\mathbb{P}_2(P)\,. 
\label{eqn:p2Prop}
\end{equation}
Since $v_h\in V_h(P)$, by definition $\Delta^2 v_h\in\mathbb{P}_2(P)$.
Consequently, we can take $\Delta^2 v_h$ as a test function $p_2$ in Equation~\eqref{eqn:p2Prop}.
Then, if we integrate by parts two times and we exploit the fact that on the boundary $v_h$ and $\partial_n v_h$ are null,
we get
$$
0 = \int_P v_h\,\Delta^2 v_h \dPP = \int_P \Delta v_h\,\Delta v_h \dPP\quad\Rightarrow\quad\Delta v_h = 0\,.
$$
Thus, since $v_h$ is zero on the boundary and its Laplacian is null, $v_h$ is the null function. 
\cvd

To set up the discrete form of Problem~\eqref{eqn:probGenWeak}, 
the projector operator $\Pi^\Delta_P$ alone is not sufficient,
we also need an $L^2-$projection operator $\Pi^0_P$ and an $H^1-$projection operator $\Pi^\nabla_P$.

Let us start with $\Pi^0_P:V_h(P)\to\mathbb{P}_2(P)$.
This projection operator is determined by the following conditions:
\begin{equation}
a^0_P(\Pi^0_P v_h,\,p_2) = a^0_P(v_h,\,p_2)\quad\forall p_2\in\mathbb{P}_2(P)\,,
\label{eqn:piZero}
\end{equation}
where we defined the bilinear form 
\begin{equation}
a^0_P(v_h,\,w_h) := \int_P v_h\,w_h\dPP\,.
\label{eqn:a0}
\end{equation}

The subsequent lemma easily follows recalling~\eqref{eqn:VirtOnPoly} and the computability of~$\Pi^\Delta_P$.

\begin{lemma}
The projection operator $\Pi^0_P:V_h(P)\to\mathbb{P}_2(P)$ is computable from $\DZ$ and $\DU$
(and actually coincides with $\Pi^\Delta_P$, see~\eqref{eqn:VirtOnPoly}).
\end{lemma}

Now we consider the projection operator $\Pi^\nabla_P:V_h(P)\to\mathbb{P}_2(P)$ defined by 
\begin{equation}
\left\{
\begin{array}{rll}
a_P^\nabla\left(\Pi^\nabla_P v_h\,, p_2\right) &=a_P^\nabla\left(v_h\,, p_2\right) &\forall p_2\in\mathbb{P}_2(P)\\[0.2em]
a_P^0\left(\Pi^\nabla_P v_h\,, 1\right) &= a_P^0\left(v_h\,, 1\right) \\
\end{array}\right.\,,
\label{eqn:piNabla}
\end{equation}
where 
\begin{equation}
a_P^\nabla\left(v_h\,,w_h\right) := \int_P \nabla\,v_h\,\cdot\nabla\,w_h\,\dPP\,. 
\label{eqn:aNabla}
\end{equation}

\begin{lemma}
The projection operator $\Pi^\nabla_P:V_h(P)\to\mathbb{P}_2(P)$ is computable from $\DZ$ and $\DU$.
\end{lemma}
\Proof We have to check that the right hand side of the first equation in~\eqref{eqn:piNabla} is computable using 
only the degrees of freedom values $\DZ$ and $\DU$.
Let us consider the first condition in~\eqref{eqn:piNabla}.
If we integrate by parts and recall definitions~\eqref{eqn:c1Enhanced} and~\eqref{eqn:VirtOnPoly},
we notice that this term depends \emph{only} on the projection operators $\Pi^0_P$ and $\Pi^\Delta_f$, 
that in turn depend only on $\DZ$ and $\DU$ values. Indeed 
\begin{eqnarray*}
a_P^\nabla\left(v_h\,, p_2\right) &=& \mathlarger{\int_P} \nabla v_h\,\cdot\nabla p_2\dPP \\ &=&
- \mathlarger{\int_P} v_h\,\Delta p_2\dPP + \int_{\partial P} v_h\,(\nabla p_2\cdot\N)\df \\  &=&
- \Delta p_2\,\mathlarger{\int_P} v_h \dPP + \sum_{f\in \partial P} \int_f v_h\,(\nabla p_2\cdot\N_f)\df \\ &=&
- \Delta p_2\,\mathlarger{\int_P} \Pi^0_P v_h \dPP + \sum_{f\in \partial P} \int_f \Pi^\Delta_f v_h\,(\nabla p_2\cdot\N_f)\df\,.
\end{eqnarray*}
\cvd

\subsection{Global virtual space $V_h(\Omega_h)$}\label{sub:genSpace}

The global discrete space 
which will be used to discretize Problem~\eqref{eqn:probGenWeak} is
\begin{equation}
V_h(\Omega_h) := \left\{v_h\in V(\Omega)\::\: v_h|_P\in V_h(P)\right\}\,.
\label{Vh-def} 
\end{equation}

Let us consider the canonical basis functions $\{\phi_i\}_i$ associated with the degrees of freedom $\DZ$ and $\DU$,
i.e. the functions $\phi_i$ which take value 1 on the $i-$th degree of freedom and vanish for the remaining ones.
It is easy to check that, 
assuming for simplicity a uniform mesh family, 
the basis functions associated with the set $\DZ$ satisfy $||\phi_i||_{L^\infty(\Om)}\sim 1$, while
the basis functions associated with $\DU$ behave like $||\phi_i||_{L^\infty(\Om)}\sim h_P$, where $h_P$
is the diameter of the polyhedron $P$.
Since this different scaling behavior with respect to the mesh size may yield detrimental effects  on the condition number 
of the discrete system it is wiser to scale accordingly the second set of degrees of freedom.

Consequently, the global degrees of freedom for $V_h(\Om)$ which we adopt in practice are
\begin{center}
\begin{minipage}{0.9\textwidth}
\begin{itemize}
 \item[$\CZ$:] evaluations of $v_h(\nu)$ at each vertex of the mesh $\Omega_h$;
 \item[$\CU$:] evaluations of $h_\nu\nabla v_h(\nu)$ at each vertex of the mesh $\Omega_h$,
\end{itemize}
\end{minipage} 
\end{center}
where $h_\nu$ denotes some local mesh size parameter, for instance the average diameter of the neighboring elements.
This choice will be better discussed in \coor{Section}~\ref{sub:sensitivity}.
The dimension of $V_h(\Om_h)$ is four times the number of mesh vertices.


\section{Discrete virtual forms and the discrete problem}\label{sec:virtForms}

When we are solving a PDE via the virtual element method, 
we have to define a suitable set of discrete forms for the problem at hand.
Such forms are constructed element-by-element and 
depend only on the local degrees of freedom $\DZ$ and $\DU$, 
also via the projection operators $\Pi^\Delta_P$, $\Pi^0_P$ and $\Pi^\nabla_P$.

Let $P\in\Omega_h$ and $v_h,w_h\in V_h(\Omega_h)$,
we define the following strictly positive bilinear form $s_P:V_h(P)\times V_h(P)\to\mathbb{R}$, 
\begin{equation}
s_P(v_h,\,w_h) := \sum_{\nu\in P} \bigg( v_h(\nu)\,w_h(\nu) + \big(h_\nu\,\nabla v_h(\nu)\big)\cdot \big(h_\nu\nabla w_h(\nu)\big)\bigg)\,,
\label{eqn:stab}
\end{equation}
where $\nu$ is a generic vertex of the polyhedron $P$ and $h_\nu$ is the scaling parameter, 
see the definition of the degrees of freedom $\CZ$ and $\CU$.
Recalling the continuous global form in Equation~\eqref{eqn:biliForms}
and the local bilinear operators defined in Equations~\eqref{eqn:aDelta},~\eqref{eqn:a0} and~\eqref{eqn:aNabla},
we construct the following local discrete linear forms
\begin{equation}
\begin{array}{rlrll}
a_{h,P}^\Delta(v_h,\,w_h) &:= a_P^\Delta(\Pi^\Delta_P v_h,\,\Pi^\Delta_P w_h) &\hspace{-0.5em}+&\hspace{-0.5em} h_P^{-1} &\hspace{-1em}s_P(v_h - \Pi^\Delta_P v_h,\,w_h - \Pi^\Delta_P w_h)\,,\\[0.2em]
a_{h,P}^\nabla(v_h,\,w_h) &:= a_P^\nabla(\Pi^\nabla_P v_h,\,\Pi^\nabla_P w_h) &\hspace{-0.5em}+&\hspace{-0.5em} h_P\,    &\hspace{-1em}s_P(v_h - \Pi^\nabla_P v_h,\,w_h - \Pi^\nabla_P w_h)\,,\\[0.2em]
a_{h,P}^0(v_h,\,w_h)      &:= a_P^0(\Pi^0_P v_h,\,\Pi^0_P w_h)                &\hspace{-0.5em}+&\hspace{-0.5em} h_P^3\,  &\hspace{-1em}s_P(v_h - \Pi^0_P v_h,\,w_h - \Pi^0_P w_h)\,,          \\[0.4em]
\end{array} 
\label{eqn:discLocalForm}
\vspace{0.5em}
\end{equation}
for all $v_h,\,w_h\in V_h(P)$, where $h_P$ is the diameter of the polyhedron $P$.
The construction above is standard in VEM, see for instance~\cite{volley,autostoppisti}.
The first term of each bilinear form in Equation~\eqref{eqn:discLocalForm} is the so-called consistency part,
while the second term is the stability part. 
This stability part is scaled in such a way that, 
under the assumptions (A1)-(A3),
there exist two positive constants $c^\star,c_\star$ such that
\begin{equation}
c_\star  a_{P}^\sharp(v_h,\,v_h) \leq a_{h,P}^\sharp(v_h,\,v_h) \leq c^\star  a_{P}^\sharp(v_h,\,v_h)
\quad\forall v_h\in V_h(P)\,,
\end{equation}
where, as before, the symbol $\sharp$ stands for $\Delta,\nabla$ and 0, respectively.

\begin{lemma}[consistency]
For all the bilinear forms in Equation~\eqref{eqn:discLocalForm} it holds
\begin{equation}
a_{h,P}^\sharp(v_h,\,p_2) =  a_{P}^\sharp(v_h,\,p_2)\quad\quad\forall p_2\in\mathbb{P}_2(P)\,, 
\forall v_h\in V_h(P)\,,
\label{eqn:consistencyProp}
\end{equation}
where the symbol $\sharp$ stands for $\Delta,\nabla$ and 0, respectively.
\label{lem:consistency}
\end{lemma}
\Proof The property in Equation~\eqref{eqn:consistencyProp} follows from the fact that 
the projection operators, $\Pi^\Delta_P$, $\Pi^\nabla_P$, and $\Pi^0_P$, are orthogonal with respect to the 
bilinear form they are associated with.
\cvd

Lemma~\ref{lem:consistency} states  
that the discrete bilinear forms $a_{h,P}^\sharp(\cdot,\cdot)$ are \emph{exact}
when one of the entries is a polynomials of degree 2.
Finally, as in a standard virtual element framework,
the global discrete forms are obtained by summing each local bilinear form over all mesh elements.
Then the discrete problem reads: find $u_h\in V_h(\Om_h)$ such that 
\begin{eqnarray}
&&c_1\,a_h^\Delta(u_h,\,v_h) + c_2\,a_h^\nabla(u_h,\,v_h) + c_3\,a^0_h(u_h,\,v_h)
= (f_h,\,v_h)_{\Om_h}\quad\forall v_h\in V_0(\Omega_h)\,.\nonumber\\
\label{eqn:probGenWeakDisc}
\end{eqnarray}
where
$$
(f_h,\,v_h)_{\Om_h} := \sum_{P\in \Om_h} \int_P \Pi^0_P v_h\,f_h\dPP\,.
$$

\coor{
\begin{remark}
The scheme of the present paper can be immediately extended to the case where the Laplace operator is used instead of the Hessian operator in the definition of the fourth order bilinear form. 
The only modification is to substitute the form $a_P^\Delta(\cdot,\cdot)$ with
$$
a_P^\Delta\left(v\,, w\right) = \int_P \Delta v \, \Delta w \dPP 
$$
and keep the same construction as in \eqref{eqn:discLocalForm} for its discrete counterpart.
\end{remark}
}


\newtheorem{thm}{Theorem}[section]
\newtheorem{cor}[thm]{Corollary}
\newtheorem{lem}[thm]{Lemma}
\newtheorem{defn}{Definition}[section]
\newtheorem{exmp}{Example}[section]

\def\P{\mathbb{P}}
\def\et{{\widetilde e}}
\def\eh{{\widehat e}}
\def\VhF{\VF}
\def\VhFn{\VFN}
\newcommand{\Wh}{W_h}

\section{Interpolation and convergence estimates}\label{sec:interpo}

In the present section we derive convergence estimates for the proposed method, under the geometric mesh assumptions (A1)-(A3) of the previous sections. In the sequel, the symbol $\lesssim$ will denote bounds up to a contant independent of $h$. 

\begin{thm}\label{thm:conv}
Let the mesh assumptions (A1)-(A3) hold.
Let $u\in H^3(\Om)$ be the solution of Problem~\eqref{eqn:probGenWeak} and 
$u_h$ the solution of the corresponding discrete formulation~\eqref{eqn:probGenWeakDisc}.
Then, it holds
$$
\| u - u_h \|_{2,\Om} \leq c\,h |u|_{3,\Om}
$$
\label{teo:estConv}
\end{thm}

To derive the proof, following the same identical steps as ~\cite{volley}, Theorem~3.1, one gets 
the ``best approximation'' bound
\begin{equation}\label{L:basic}
\| u - u_h \|_{2,\Omega} \lesssim \| u - u_I \|_{2,\Omega} 
+ \| u - u_\pi \|_{2,\Omega_h} + h^2 \| f \|_{0,\Omega} ,
\end{equation}
for any interpolant $u_I \in V_h$ and piecewise $\P_2$-polynomial $u_\pi$, and
where $| \cdot |_{s,\Omega_h}$ denotes a broken (with respect to the mesh) Sobolev norm of order $s$, $s\ge0$. 

The second term is immediately bounded by standard polynomial approximation estimates on star-shaped domains (see for instance~\cite{Brenner:Scott}), yielding
$$
\| u - u_\pi \|_{2,\Omega_h} \lesssim h | u |_{3,\Omega_h} .
$$

Therefore, the main effort in proving Theorem~\ref{teo:estConv} is bounding the first term in the right hand side of~\eqref{L:basic}, that is showing the interpolation estimates for the space $V_h$. In order to do so, we will first prove interpolation estimates for the simpler space
$$
W_h(\Om_h) := \{v_h\in V(\Om)\::\: v_h|_P\in\Wh(P)\quad\forall P\in\Om_h\,\}\,,
$$
where
\begin{eqnarray*}
\Wh(P) :=  \bigg\{v_h\in H^2(P)\:&:&\: \Delta^2\,v_h=0, \nonumber \\[-0.3em]
\:&\phantom{:}&\: 
v_h|_{S_P}\in C^0(S_P)\,,
\nabla v_h|_{S_P}\in[C^0(S_P)]^3\,,\nonumber \\
\:&\phantom{:}&\: 
v_h|_f\in \VF\,,\,
\partial_{n_f} v_h|_f\in \VFN
\quad\forall f\in\partial P
\bigg\}\,,\nonumber \\
\label{eqn:defWh}
\end{eqnarray*}

Following the same arguments in Section~\ref{sub:VSOnPoly}, it is easy to check that the operators $\DZ$ and $\DU$
constitute a set of degrees of freedom also for $\Wh(P)$.

\begin{remark}
By adding and subtracting a piecewise second-order polynomial, then using a triangle inequality and the continuity of $\Pi^\Delta_P$ in the $H^2$ norm, finally recalling standard approximation results for polynomials on star-shaped domains, from Theorem \ref{thm:conv} one can easily derive also
$$
\sum_{P \in \Omega_h} \coor{\| u - \Pi^\Delta_P(u_h) \|_{2,P}^2 \leq c\,h^2 |u|_{3,\Om_h}^2}
$$
that states the convergence of the projected discrete solution.
\end{remark}

\subsection{Interpolation estimates for $\Wh$}

In deriving the estimates for the space $\Wh$, we will take full advantage of known results for two-dimensional $C^0$ and $C^1$ VEM spaces (cited below). In addition, we will use the following standard results on the continuous dependence of the solution on the boundary biharmonic data in a polyhedron $P$ \coor{(see for instance~\cite{Girault-Raviart:book1986, Barton-Mayboroda})}.

Given a polyhedron $P$, let $r_1,r_2$ be two scalar functions living on $\partial P$ satisfying $r_1 \in C^0(\partial P)$ and $r_1 \in H^{3/2}(f)$, $r_2 \in H^{1/2}(f)$ for each face $f \in \partial P$. Consider the standard biharmonic Dirichlet problem
\begin{equation}
\left\{
\begin{aligned}
 \Delta^2 v = 0     &\qquad & \textrm{in } P \\
 v = r_1            &\qquad & \textrm{on } \partial P\\
 \partial_n v = r_2 &\qquad & \textrm{on } \partial P
\end{aligned}
\right.
\end{equation}
where all the operators are to be intended in weak sense. Below, $\N$ will denote the outward normal to the polyhedron's boundary (face by face).

\begin{lem}\label{L:ghostlemma}
Let the auxiliary three-dimensional vector field ${\bf r} = \nabla_\tau r_1 + {\bf n} \, r_2$ living on $\partial P$. Assume that such vector field ${\bf r}$ is (component-wise) in $H^{1/2}(\partial P)$. Then it holds 
$$
| u |_{2,P} \le C |{\bf r}|_{1/2,\partial P} .
$$
The constant $C$ here above depends only on the star-shapedness of the polyhedron (the constant appearing in assumption (A1)-(A2)) and the Lipschitz constant of its boundary.
\end{lem}

Note that the condition ${\bf r} \in [H^{1/2}(\partial P)]^3$ takes
into account the necessary compatibility conditions at the edges.
We can now state the following interpolation result for the $\Wh$ space. 

\begin{prop}\label{L:wh}
Let $u \in H^3(\Omega)$ and $w_I$ the only function in $\Wh$ that interpolates the nodal values of $u$ and $\nabla u$ at all vertexes of $\Omega_h$. Then it holds
$$
| u - w_I |_{2,\Omega} \lesssim h |u|_{3,\Omega_h} .
$$
\end{prop}
\begin{proof}
We prove a local interpolation estimate, the global one following immediately by summing over all the elements. 
Let $P \in \Omega_h$. We start by splitting the error $u - w_I = \et + \eh$ where
$$
\begin{aligned}
& \Delta^2 \et = 0 \textrm{ in } P , \ \ \et = u-w_I \textrm{ on } \partial P, \ \
\partial_n \et = \partial_n (u-w_I) \textrm{ on } \partial P , \\
& \Delta^2 \eh = \Delta^2 (u-w_I) \textrm{ in } P , \ \ \eh = 0 \textrm{ on } \partial P, \ \
\partial_n \eh = 0 \textrm{ on } \partial P . \\
\end{aligned}
$$ 
An integration by parts easily shows that
\begin{equation}\label{L:split}
| u - w_I |_{2,P}^2 = | \eh |_{2,P}^2 + | \et |_{2,P}^2 ,
\end{equation}
so that we need to bound the two terms above.
For the first term, we again integrate by parts twice and obtain, also recalling that 
$\Delta^2 w_I = 0$ by definition of \coor{$W_h(P)$},
\begin{eqnarray*}
| \eh |_{2,P}^2 &=& \int_P (\Delta^2 \eh) \, \eh\dPP \le \| \Delta^2 \eh \|_{-1,P} \, \| \eh \|_{1,P}\\
&=& \| \Delta^2 u \|_{-1,P} \, \| \eh \|_{1,P} \lesssim \| \Delta^2 u \|_{-1,P} \, | \eh |_{1,P} ,
\end{eqnarray*}
where in the last step we used a Poincar\'e inequality ($\eh$ vanishes on the boundary of $P$).
The first multiplicative term in the right hand side is bounded by $| u |_{3,P}$ (to show this it is sufficient to apply the definition of dual norm and integrate once by parts). The second term corresponds to the $L^2$ norm of $\nabla\eh$, that is a (vector valued) function in $H^1(P)$ vanishing on the boundary (see definition of $\eh$). Therefore a scaled Poincar\'e inequality immediately yields 
\begin{equation}\label{L:bound1}
| \eh |_{2,P}^2 \le | u |_{3,P} \, \| \nabla \eh \|_{0,P} \lesssim h_P | u |_{3,P} \, | \eh |_{2,P} 
\end{equation}
that gives the desired bound for the first term in \eqref{L:split}.

For the second term in~\eqref{L:split}, we make use of Lemma~\ref{L:ghostlemma} and the definition of $\et$. Note that, due to the regularity of $u$ and the definition of \coor{$W_h(P)$}, the boundary data in the definition of $\eh$ satisfies the hypotheses of the Lemma. Moreover, it is trivial to check that the vector field ${\bf r}$ appearing in Lemma~\ref{L:ghostlemma} in this case is nothing but $\nabla (u-w_I)$. Therefore we obtain the bound
\begin{equation}\label{L:eq1}
| \et |_{2,P}^2 \lesssim | \nabla (u-w_I) |_{1/2,\partial P}^2 .
\end{equation}

Note that the above bound is uniform (in $P$) since the elements $P$ are star shaped and have uniformly bounded Lipschitz constant. Indeed, the observation in Remark~\ref{L:tets} easily implies that each polyhedron $P$ has a uniformly Lipschitz continuous boundary (actually, it holds also under the assumption (A1) alone, as shown in~\cite{BOTTI2018278}).

By definition of the face spaces $\VhF$ and $\VhFn$, the components of the vector field $\nabla w_I$ are in $H^1(f)$ for every face $f \in \partial P$. Since by definition of $\Wh(P)$ the gradient of $w_I$ is continuous on the skeleton, we have $\nabla w_I \in [H^1(\partial P)]^3$. Standard trace estimates, recalling $u \in H^3(P)$ imply an analogous property $\nabla u \in [H^1(\partial P)]^3$. Therefore, first by space interpolation theory and then summing on faces, from~\eqref{L:eq1} we get
\begin{equation}\label{L:eq1}
\begin{aligned}
| \et |_{2,P}^2 & \lesssim \| \nabla (u-w_I) \|_{0,\partial P} | \nabla (u-w_I) |_{1,\partial P}  \\
& \lesssim \Big( \sum_{f \in \partial P} \| \nabla (u-w_I) \|_{0,f}^2 \Big)^{1/2} 
\Big( \sum_{f \in \partial P} | \nabla (u-w_I) |_{1,f}^2 \Big)^{1/2} .
\end{aligned}
\end{equation}
We note that, in both terms above, one can split for each face $f$
$$
\nabla (u-w_I)|_f = (\nabla (u-w_I)|_f)_{\boldsymbol \tau} + (\nabla (u-w_I)|_f \cdot \N_f)\N_f\,,
$$
that is the tangential and normal components of the vector $\nabla (u-w_I)|_f$. 
Therefore, for every face $f$
\begin{equation}\label{L:eq2}
\begin{aligned}
\| \nabla (u-w_I) \|_{0,f} 
& \le \| (\nabla (u-w_I)|_f)_{\boldsymbol \tau} \|_{0,f} + \| \partial_n (u-w_I) \|_{0,f} \\
& = | u-w_I |_{1,f} + \| \partial_n u - \partial_n w_I \|_{0,f}
\end{aligned}
\end{equation}
and analogously
\begin{equation}\label{L:eq3}
| \nabla (u-w_I) |_{1,f} 
\le | u-w_I |_{2,f} + | \partial_n u - \partial_n w_ I|_{1,f} .
\end{equation}
We now need to recall that the restriction to faces of the space $\Wh(P)$ corresponds, by definition, to $C^1$ virtual spaces in 2D~\cite{Brezzi:Marini:plates,BM13,ABSVu} and that its normal derivative corresponds to $C^0$ virtual spaces in 2D~\cite{volley, projectors}.
Therefore the bounds for the first term in the right-hand side of~\eqref{L:eq2} and for the first term in the right-hand side of~\eqref{L:eq3} follow from known interpolation theory for $C^1$ virtual spaces in 2D, see~\cite{Paper_con_Gonzalo}. The bounds for the second term in the right-hand side of~\eqref{L:eq2} and for the second term in the right-hand side of~\eqref{L:eq3} follow from known interpolation theory for $C^0$ virtual spaces in 2D, see \cite{BLR-stab,Brenner}. Therefore, from \eqref{L:eq1} combined with~\eqref{L:eq2}-\eqref{L:eq3}, we get
\begin{equation}\label{L:bound2}
\begin{aligned}
| \et |_{2,P}^2 & \lesssim 
\Big( \sum_{f \in \partial P}  h_f^3 |u|_{5/2,f}^2 + h_f^3 |\partial_n u|_{3/2,f}^2 \Big)^{1/2} 
\Big( \sum_{f \in \partial P} h_f |u|_{5/2,f}^2 + h_f |\partial_n u|_{3/2,f}^2 \Big)^{1/2} \\
& \lesssim h_P^2 |u|_{3,P}^2 ,
\end{aligned}
\end{equation}
where the last bound above follows from a (face by face) trace inequality.
The local result now follows easily combining~\eqref{L:split} with~\eqref{L:bound1} and~\eqref{L:bound2}
\begin{equation}\label{L:loc}
|u - w_I |_{2,P} \lesssim h_P |u|_{3,P} \qquad \forall P \in \Omega_h.
\end{equation}
\end{proof}

\subsection{Interpolation estimates for $V_h$}

We have the following result.

\begin{prop}\label{L:vh}
Let $u \in H^3(\Omega)$ and $u_I$ the only function in $V_h$ that interpolates the nodal values of $u$ and $\nabla u$ at all vertexes of $\Omega_h$. Then it holds
$$
| u - u_I |_{2,\Omega} \lesssim h |u|_{3,\Omega_h} .
$$
\end{prop}
\begin{proof}
Given \coor{$u \in H^3(\Omega)$}, let $w_I$ be its interpolant in $\Wh$. We fix our attention on a generic polyhedron $P\in\Omega_h$, the global estimates will then follow from the local ones by summing over all elements.
Moreover let the auxiliary space $Q = \P_2(P)$. 
Now we consider the following problem in mixed form 
\begin{equation}\label{L:rhopbl}
\left\{
\begin{aligned}
& \textrm{Find } \varphi \in H^2_0(P), \ p \in Q \textrm{ such that}  \\
& \int_P \nabla^2 \varphi : \nabla^2 v \dPP + \int_P p\,v \dPP= 0 \qquad \forall v \in H^2_0(P) \\
& \int_P \varphi\,q \dPP = \int_P \big( \Pi^\Delta_P (w_I) - w_I \big) q \dPP \qquad \forall  q \in Q .
\end{aligned}
\right.
\end{equation}
We endow the space $Q$ with the norm
$$
\| q \|_Q := h_P^2 \| q \|_{0,P} \quad \forall q \in Q .
$$
Problem~\eqref{L:rhopbl} is a standard problem in mixed form. Since the coercivity on the kernel is clearly guaranteed, in order to prove its well posedness we need only to check the inf-sup condition (see for instance~\cite{Bo-Bre-For}). Given any $q \in Q$, let $T$ be any one of the tetrahedra of Remark~\ref{L:tets} and let $b_T$ be the standard quartic bubble on $T$. Then, noting that $b_T^2 q \in H^2_0(P)$, standard properties of polynomials yield 
$$
\sup_{v \in H^2_0(P)} \frac{\int_P q\,v\dPP}{| v |_{2,P}}
\ge \frac{\int_T q \, (b_T^2 q)\dPP}{| b_T^2 q |_{2,T}} \gtrsim
\frac{\|q\|^2_{0,T}}{h_P^{-2} \| q \|_{0,T}} = h_P^2\| q \|_{0,T} \gtrsim h_P^2\| q \|_Q ,
$$
that is the inf-sup condition for problem~\eqref{L:rhopbl}.
Since~\eqref{L:rhopbl} is well posed, we have
\begin{eqnarray*}
| \varphi |_{2,P} &\lesssim&  \|   \Pi^\Delta_P (w_I) - w_I  \|_{Q^\star}
= \sup_{q \in Q} \frac{\int_P \big( \Pi^\Delta_P (w_I) - w_I \big)\,q\dPP}{h_P^2 \| q \|_{0,P}}\\
&\le& h_P^{-2} \|   \Pi^\Delta_P (w_I) - w_I  \|_{0,P} .
\end{eqnarray*}
Since by definition see~\eqref{eqn:piDelta}
$$
\int_{\partial P} \Big( \Pi^\Delta_P (w_I) - w_I \Big) p_1\dpPP = 0\qquad\forall p_1 \in \P_1(P)\,,
$$
by a Poincar\'e-type inequality (the standard proof being omitted for the sake of brevity)
the above bound becomes
$$
| \varphi |_{2,P} \lesssim |   \Pi^\Delta_P (w_I) - w_I  |_{2,P} .
$$
By a triangle inequality (and recalling that the operator $\Pi^\Delta_P$ is a projection operator onto $\P_2(P)$ minimizing the distance in the $H^2$ seminorm) the above bound  leads to
\begin{equation}
\begin{aligned}
| \varphi |_{2,P} & \lesssim |   \Pi^\Delta_P (w_I - u)  |_{2,P}
+ |\Pi^\Delta_P (u) - u  |_{2,P} + |   u - w_I  |_{2,P} \\
& \le |\Pi^\Delta_P (u) - u  |_{2,P} + 2 |  u - w_I  |_{2,P} .
\end{aligned}
\end{equation}
The first term above is bounded by standard polynomial approximation, while the second one  is bounded using \eqref{L:loc}. We get
\begin{equation}\label{L:rho-bound}
| \varphi |_{2,P} \lesssim h_P |  u  |_{3,P} .
\end{equation}
We are now ready to present the interpolant in the $V_h$ space, that is $u_I = w_I + \varphi$. 
We first check that $u_I \in V_h$, see definition~\eqref{eqn:VirtOnPoly}. 
\begin{itemize}
\item $u_I$ satisfies the conditions at the boundary since $\varphi$ and $\partial_n \varphi$ vanish at $\partial P$;
\item $\Delta^2 u_I \in \P_2(P)$ since $\Delta^2 w_I = 0$ and we deduce that $\Delta^2 \varphi = -p~\in~\P_2(P)$ from the first equation of~\eqref{L:rhopbl};
\item It is easy to check that, by definition of $\Pi^\Delta_P$ and integrating by parts, it holds $\Pi^\Delta_P(\varphi)=0$. Therefore, using the second equation of~\eqref{L:rhopbl}, it immediately follows that, 
for any $q\in \P_2(P)$,
$$
\int_P u_I\,q\dPP = \int_P (w_I + \varphi)\,q\dPP = \int_P \Pi^\Delta_P(w_I)\,q\dPP = \int_P \Pi^\Delta_P(u_I)\,q\dPP\,.
$$
\end{itemize}
Therefore $u_I \in V_h$ since it satisfies all conditions in the definition. 
Finally, the result follows from~\eqref{L:loc} and~\eqref{L:rho-bound}
$$
|u - u_I |_{2,P} \le  |u - w_I |_{2,P} + | \varphi |_{2,P} 
\lesssim h_P |u|_{3,P} .
$$
\end{proof}

\begin{cor}
Let $u \in H^3(\Omega)$ and $u_I$ the only function in $V_h$ that interpolates the nodal values of $u$ and $\nabla u$ at all vertexes of $\Omega_h$. Then it holds
$$
| u - u_I |_{m,\Omega} \lesssim h^{3-m} |u|_{3,\Omega_h}\quad \textrm{ for } m=0,1 .
$$
\end{cor}
\begin{proof}
Let $U_h$ denote the standard $C^0$ virtual element space of order 1 in 3D (see for instance \cite{projectors,apollo11}). The degrees of freedom of such space are simply given by the value at all mesh vertexes. Let $\psi_I \in U_h$ be the unique vertex interpolant of $(u - u_I)$. Since $u-w_I$ vanishes at all vertexes, $\psi_I=0$. Therefore, also using approximation estimates for $C^0$ virtual element spaces in 3D (see for instance \cite{Cangiani:apos,Brenner}), we get
$$
| u - u_I |_{m,\Omega} = | (u - u_I) - \psi_I |_{m,\Omega} \lesssim h^{2-m} | u - u_I |_{2,\Omega} 
$$
The result follows using Proposition \ref{L:vh}.
\end{proof}

\coor{
\begin{remark}
The above results could be easily extended to the case with lower regularity $u \in H^{s}$, $s > 5/2$, since in such case the above interpolants are still well defined. Instead, extending to the case $u \in H^{s}$ with $2 < s < 5/2$ would require a different kind of interpolation (in the Cl\'ement or Scott-Zhang spirit). 
\end{remark}
}


\section{Numerical results}\label{sec:numExe}

In this section we numerically validate the theory behind the $C^1$ virtual elements proposed in this paper.

\subsection{Domain discretization}

We will consider two different computational domains:
the standard unit cube $[0,\,1]^3$ and the truncated octahedron~\cite{patagonGeo}.
We discretize such geometries in three different ways.

\begin{itemize}
 \item \textbf{Structured}: the computational domain is decomposed by cubes inside the domain and 
 arbitrary shaped polyhedron close to the boundary, see Figure~\ref{fig:pataDisc} (a).
 When we take the unit cube as domain,
 this type of mesh becomes a standard structured decomposition composed by small cubes.
 \coor{\item \textbf{Tetra}: a Delaunay tetrahedral mesh of the input domain, see Figure~\ref{fig:pataDisc}~(b).}
 \item \textbf{CVT}: the domain is discretized via a Centroidal Voronoi Tessellation, i.e. 
 a Voronoi tessellation where the centroid of the Voronoi cells coincides with the control points of the cells.
 This kind of mesh can be computed via a standard Lloyd algorithm~\cite{cvtPaper}.
 In Figure~\ref{fig:pataDisc} (c) we show a CVT discretization of the truncated octahedron geometry.
 \item \textbf{Random}: refers to a Voronoi tessellation 
 where we randomly distributed the control points of the cells inside the domain and 
 we do not make any optimization on cells' shape, see Figure~\ref{fig:pataDisc} (d).
\end{itemize}

\begin{figure}[!htb]
\begin{center}
\begin{tabular}{cc}
\includegraphics[width=0.45\textwidth]{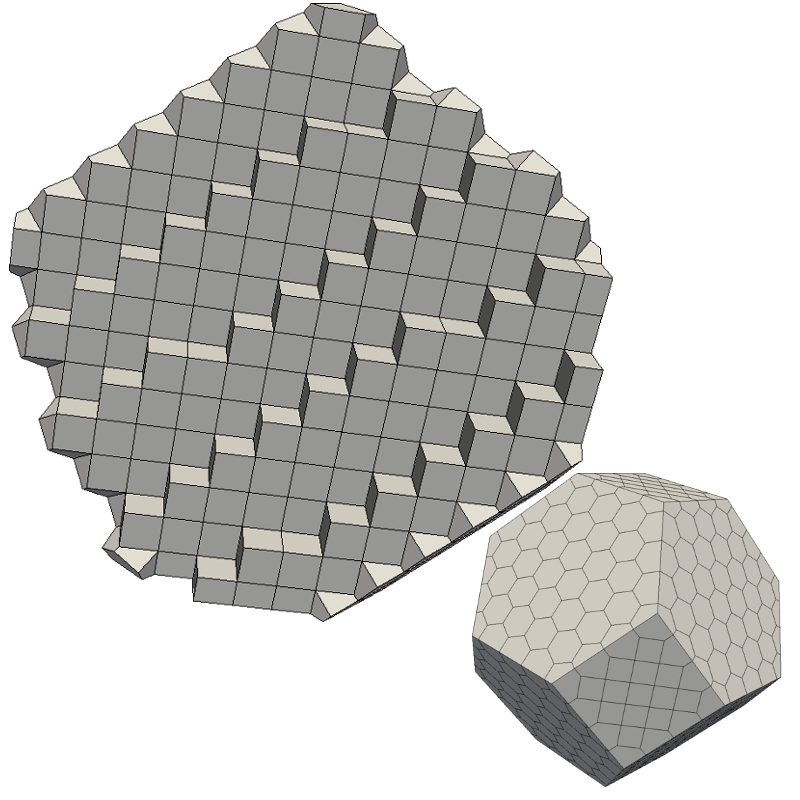} &
\includegraphics[width=0.45\textwidth]{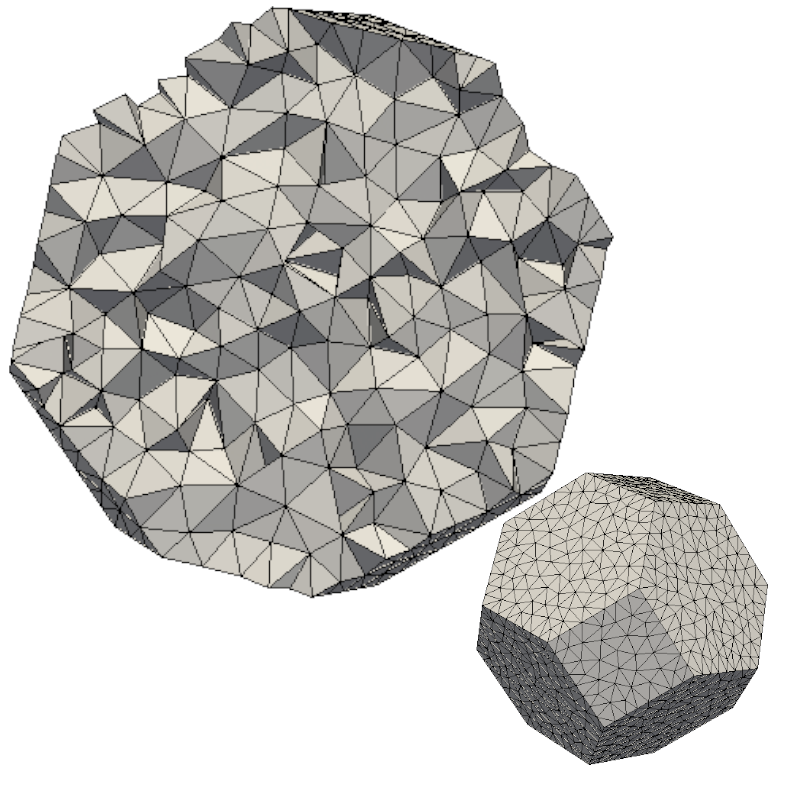} \\
(a) & (b) \\
\includegraphics[width=0.45\textwidth]{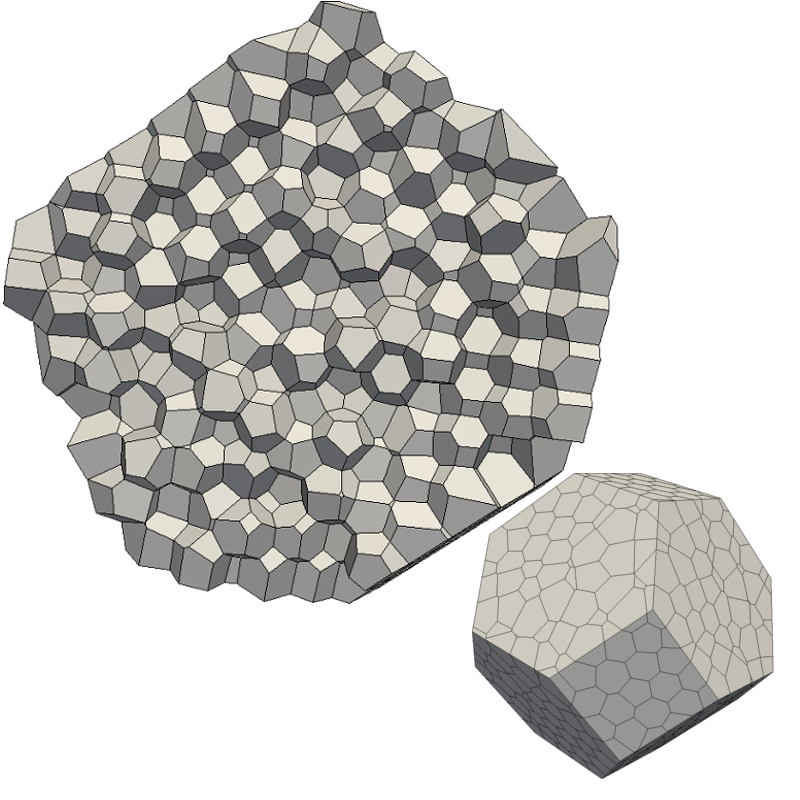} & 
\includegraphics[width=0.45\textwidth]{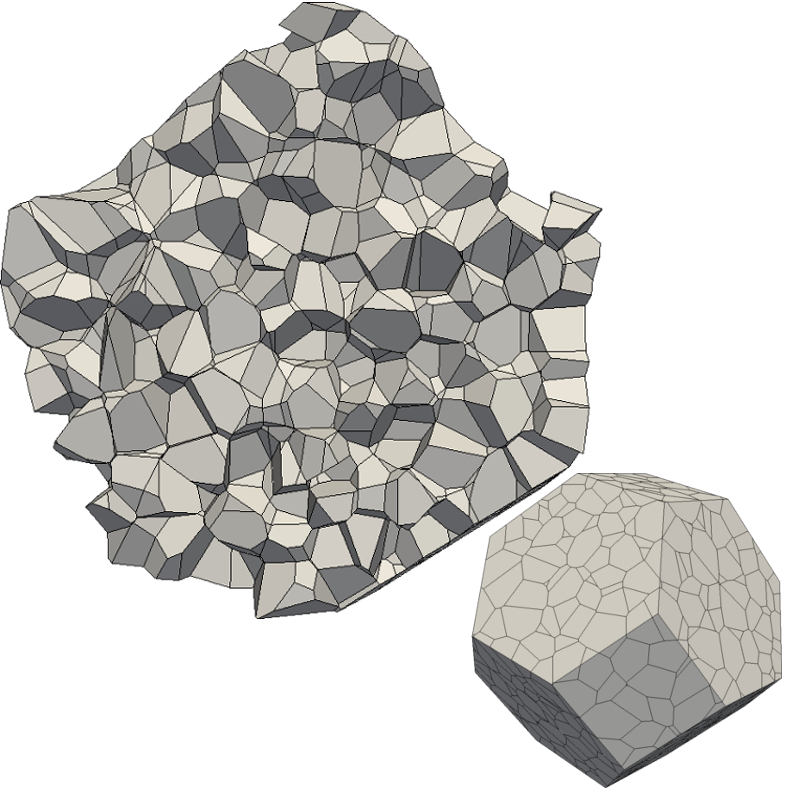}\\
(c) & (d) \\
\end{tabular}
\end{center}
\caption{\coor{Truncated octahedron geometries with Structured (a), tetrahedral (b), CVT (c) and Random (d) discretization.}}
\label{fig:pataDisc}
\end{figure}
    
In order to build such meshes we exploit the c++ library \texttt{voro++}~\cite{voroPlusPlus} and
follow the strategies described in~\cite{apollo11}, 
\coor{while for tetrahedral meshes we use \texttt{tetgen}~\cite{tetgen}.}
We construct a sequence of meshes for each type and 
we define the mesh-size as 
$$
h := \frac{1}{n_P}\sum_{P\in\Omega_h} h_P\,,
$$
where $n_P$ is the number of polyhedrons $\Omega_h$.

We underline that the \textbf{Random} partitions are particularly interesting from the computational point of view. 
Indeed, such meshes contains small edges/faces and stretched polyhedrons so 
the robustness of the virtual element method will be severely tested.

\subsection{Error norms}

Suppose that $u$ is the exact solution of the partial differential equation we are solving and 
let $u_h$ be the discrete solution provided by VEM. 
We consider the following error quantities:
\begin{itemize}
 \item \textbf{$H^2$-seminorm relative error:} 
 $$
 e_{H^2} := \frac{1}{|u|_{2,\Omega}}\left(\sum_{P\in\Omega}^{n_P} |u-\Pi^\Delta_P u_h|_{2,P}^2\right)^{1/2}\,,
 $$
 where $\Pi^\Delta_P$ is the $\Delta-$projection operator defined in Equation~\eqref{eqn:piDelta};
 \item \textbf{$H^1$-seminorm and $L^2$-norm relative errors:} 
 \begin{eqnarray*}
 e_{H^1} &:=& \frac{1}{|u|_{1,\Omega}} \left(\sum_{P\in\Omega}^{n_P} |u-\Pi^\nabla_P u_h|_{1,P}^2\right)^{1/2} \, ,\\[1em]
 \qquad e_{L^2} &:=& \frac{1}{||u||_{2,\Omega}}\left(\sum_{P\in\Omega}^{n_P} ||u-\Pi^0_P u_h||_{2,P}^2\right)^{1/2}\,, 
 \end{eqnarray*}
where $\Pi^\nabla_P$ and $\Pi^0_P$ are the operators defined in Equations~\eqref{eqn:piNabla}
and~\eqref{eqn:piZero}, respectively;
 \item \textbf{$l^\infty$-type relative error:} 
 We consider the $l^\infty$-type error for functions and gradients 
 \begin{eqnarray*}
 e_{l^\infty} &:=& \frac{\max_{\nu \in \Omega_h}|u(\nu)-u_h(\nu)|}{\max_{\nu \in \Omega_h}|u(\nu)|}\\[1em]
 e_{l^\infty}^\nabla &:=& \frac{\max_{\nu \in \Omega_h}||\nabla u(\nu) - \nabla u_h(\nu)||_\infty}{\max_{\nu \in \Omega_h}||\nabla u(\nu)||_\infty}\,,  
 \end{eqnarray*}

 where $||\cdot||_\infty$ is the standard $l^\infty$ norm of three dimensional vectors.
 
\end{itemize}

\subsection{Numerical experiments}

In the following three \coor{Section}s we develop three different experiments to validate the proposed method.
First of all we show a convergence analysis of the method, \coor{Section}~\ref{sub:hConv}.
Then, we analyze different choices of the scaling parameter $h_\nu$, \coor{Section}~\ref{sub:sensitivity}.
Finally we compare this method with the standard $C^0$ VEM approach proposed in~\cite{apollo11}, 
\coor{Section}~\ref{sub:c0Comp}.

\coor{The numerical scheme was developed inside the \texttt{vem++} library, a \texttt{c++} code built at the University of Milano - Bicocca during the CAVE project \\ (https://sites.google.com/view/vembic/home).}

\subsection{Example 1: Bi-Laplacian with reaction, $h-$convergence analysis}\label{sub:hConv}

Let $\Omega$ be the truncated octahedron,
we consider the following partial differential equation 
\begin{equation}
\left\{
\begin{array}{rlll}
\Delta^2 u + u &=&\, f\quad\quad&\textnormal{in }\Omega\\
u &=&\,g_1 \quad\quad&\textnormal{on }\partial\Omega\\
\partial_n u &=&\,g_2  \quad\quad&\textnormal{on }\partial\Omega \\
\end{array}
\right.,
\label{eqn:reacDiff}
\end{equation}
the right hand side $f$ and the Dirichlet boundary conditions $g_1$ and $g_2$ are chosen in such a way that the 
solution of Equation~\eqref{eqn:reacDiff} is 
$$
u(x,\,y,\,z):= \sin(\pi xyz)\,.
$$
In the present test we take as parameter $h_\nu$ the mean value of the diameters of all polyhedrons 
which share the mesh vertex $\nu$, 
a quite natural choice, see \coor{Section}~\ref{sub:genSpace}.

In Figure~\ref{fig:Ese1} we show the convergence lines for the errors in the $H^2$-seminorm and $L^2$-norm.
The trend of the $H^2$-seminorm error is the expected one, see Theorem~\ref{teo:estConv}, 
i.e. it is approximately of order 1.
The convergence lines of the each type of mesh are close to each other.

\begin{figure}[!htb]
\begin{center}
\begin{tabular}{cc}
\includegraphics[width=0.45\textwidth]{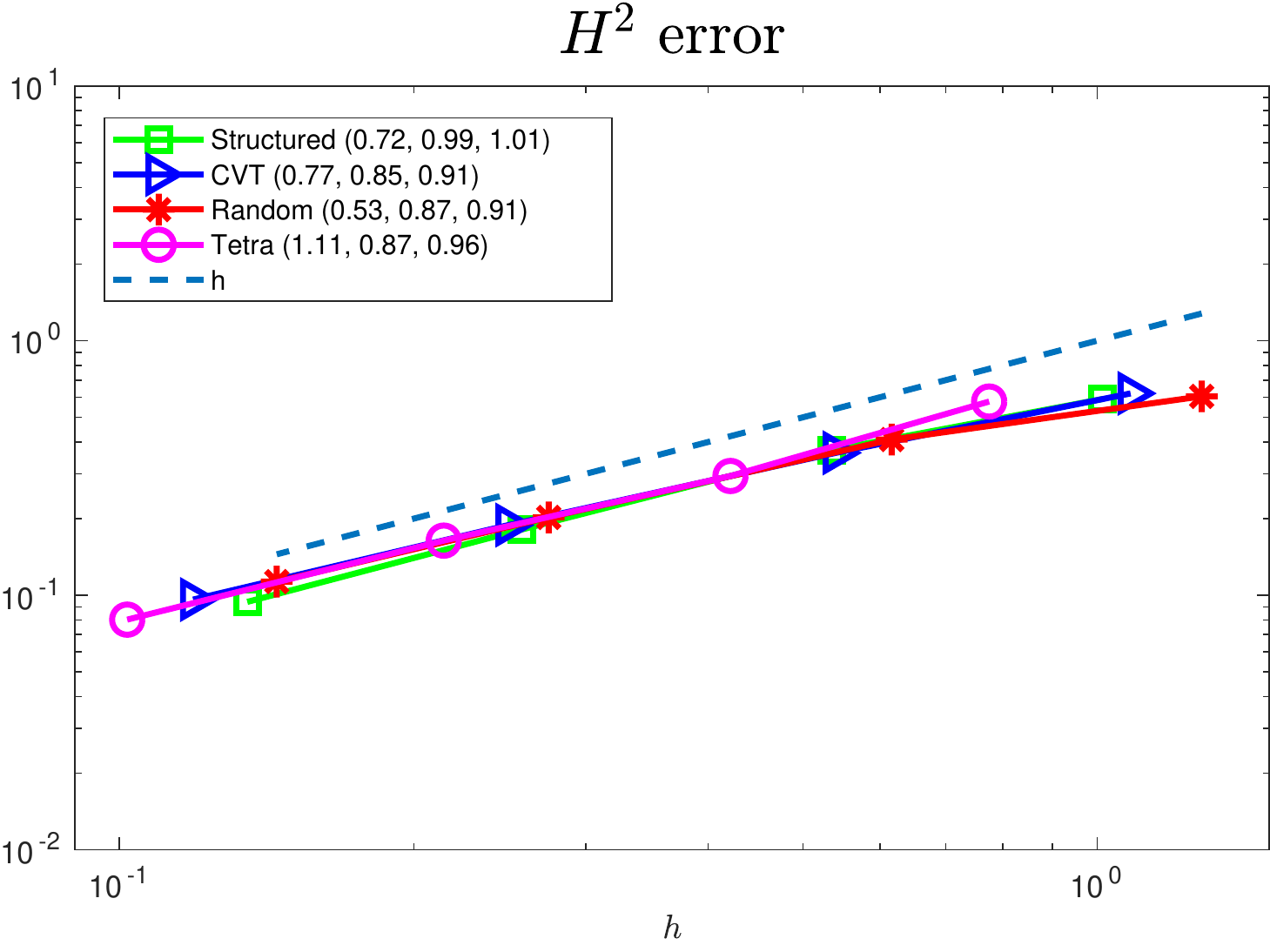} & 
\includegraphics[width=0.45\textwidth]{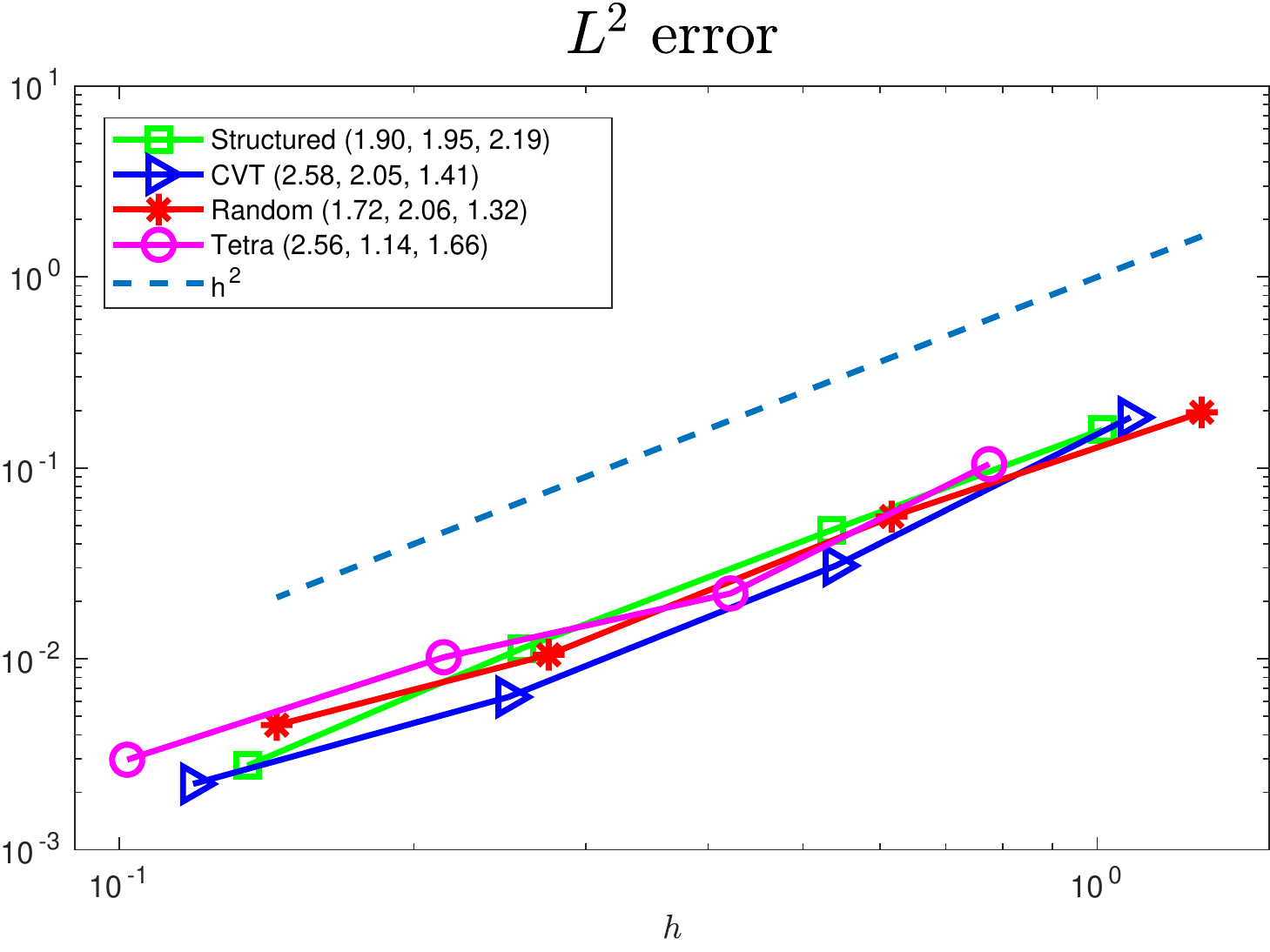}   \\
\end{tabular}
\end{center}
\caption{\coor{Example 1: convergence lines of $e_{H^2}$ (left) and $e_{L^2}$ (right) for the \textbf{Structured}, \textbf{Tetra}, \textbf{CVT} and \textbf{Random}.
In the legend we report the \coor{convergence order} at each step.}}
\label{fig:Ese1}
\end{figure}

In Figure~\ref{fig:Ese1LInfty} we show the trend of the errors $e_{l^\infty}$ and $e_{l^\infty}^\nabla$.
We did not derive a theoretical proof about the trend of such errors, but we can empirically deduce from these graphs that 
the convergence rate of $e_{l^\infty}$ is between 2 and 3 
while the one of $e_{l^\infty}^\nabla$ is 2.

\begin{figure}[!htb]
\begin{center}
\begin{tabular}{cc}
\includegraphics[width=0.45\textwidth]{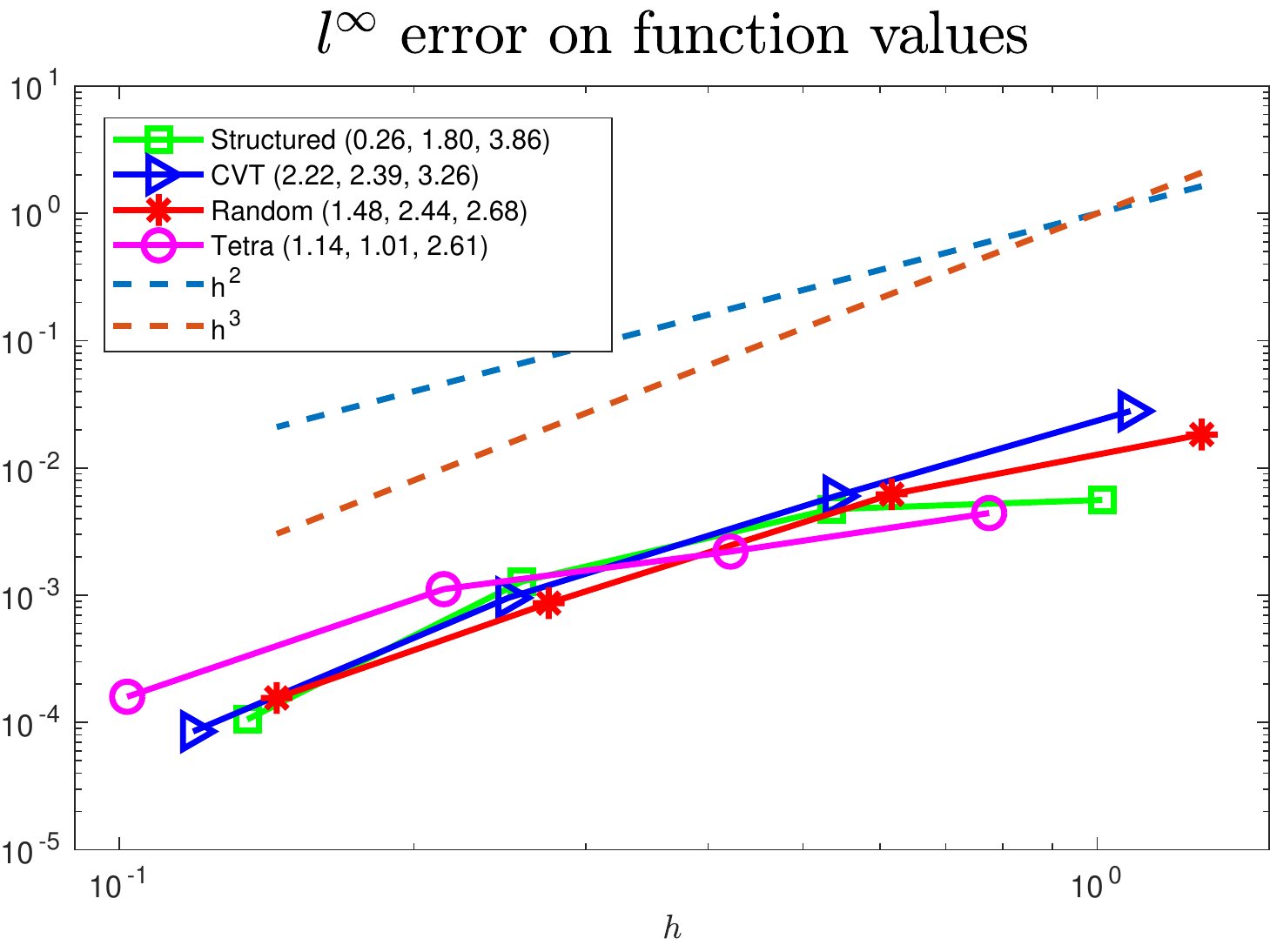} & 
\includegraphics[width=0.45\textwidth]{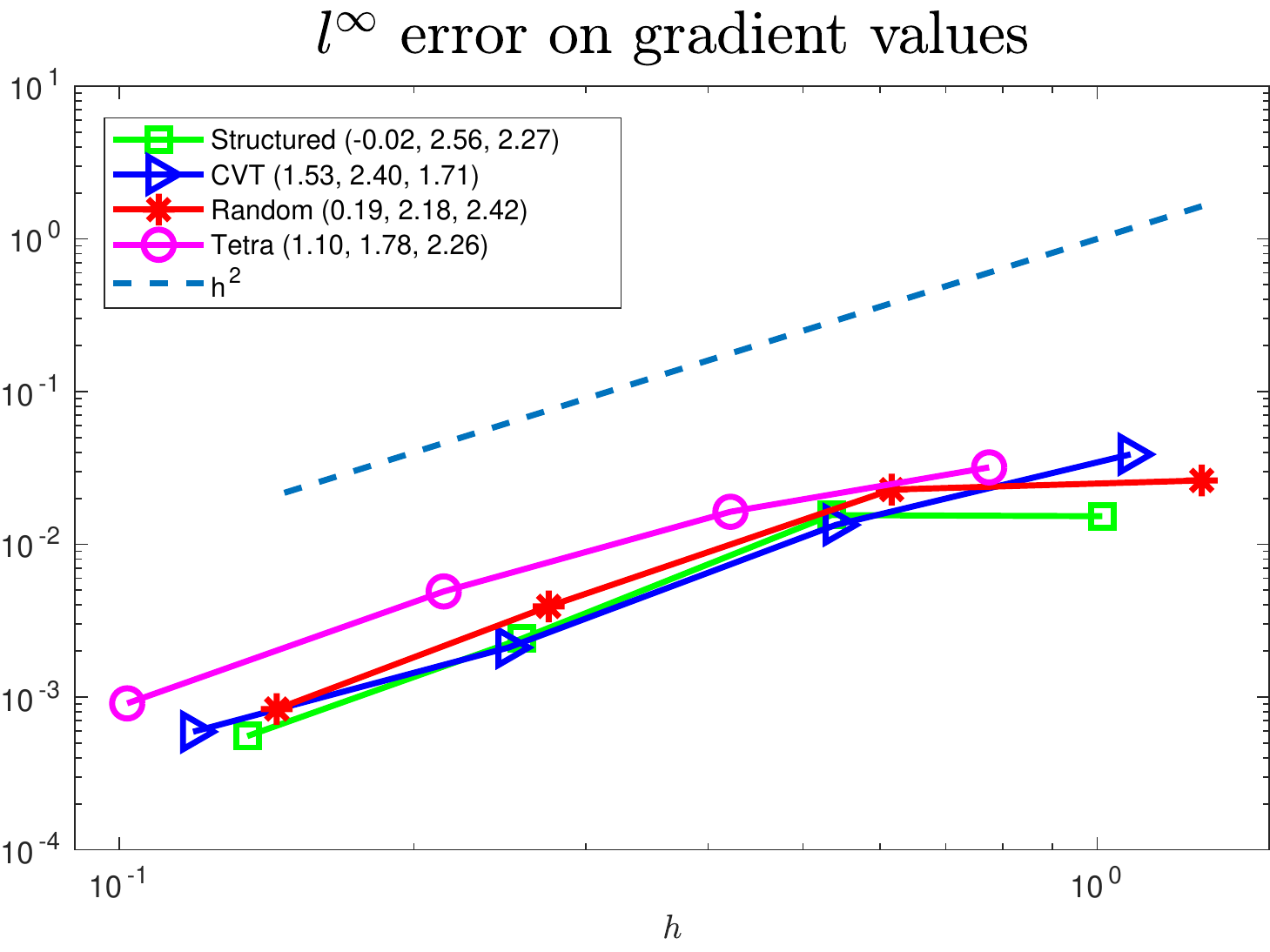}   \\
\end{tabular}
\end{center}
\caption{\coor{Example 1: convergence lines of $e_{l^\infty}$ (left) and $e_{l^\infty}^\nabla$ (right) for the \textbf{Structured}, \textbf{Tetra}, \textbf{CVT} and \textbf{Random}. 
In the legend we report the \coor{convergence order} at each step.}}
\label{fig:Ese1LInfty}
\end{figure}

\subsection{Example 2: analysis on $h_\nu$}\label{sub:sensitivity}

In the present section we investigate different choices of the ``local mesh size'' scaling parameter $h_\nu$ 
introduced in \coor{Section}~\ref{sub:genSpace}, see in particular Equation~\eqref{eqn:stab}.
We consider the following problem
\begin{equation}
\left\{
\begin{array}{rlll}
\Delta^2 u &=&\, f\quad\quad&\textnormal{in }\Omega\\[0.5em]
u &=&\,g_1 \quad\quad&\textnormal{on }\partial\Omega\backslash\Gamma\\
\partial_n u &=&\,g_2  \quad\quad&\textnormal{on }\partial\Omega\backslash\Gamma \\[0.5em]
\Delta u &=& 0 \quad\quad& \textnormal{on }\Gamma\\
-\partial_n\Delta u &=& 0 \quad\quad&\textnormal{on }\Gamma \\
\end{array}
\right.,
\label{eqn:neudiri}
\end{equation}
where $\Omega$ is the standard unit cube, 
$\Gamma$ are the faces associated with the planes $x=0$ and $x=1$,
where we apply homogeneous Neumann boundary conditions,
$f$, $g_1$ and $g_2$ are chosen in such a way that the exact solution is the function 
$$
u(x,\,y,\,z) = \frac{1}{12}\,x^4y\,z\,.
$$

Before showing the numerical results,
we explain the choices we made for the scaling parameter $h_\nu$.
Given a vertex $\nu$ of a mesh, we temporary use the following labels to denote these three collections of diameters:
\begin{itemize}
 \item $h_P$ the diameter of all polyhedrons sharing $\nu$;
 \item $h_f$ the diameter of all faces sharing $\nu$;
 \item $h_e$ the diameter of all edges whose endpoint is $\nu$.
\end{itemize}
Then we can take the mean, the maximum or the minimum of these set of diameters
to associate with $\nu$ a unique scalar value $h_\nu$.
These operations give a total of $3\times 3=9$ possible choices for $h_\nu$.
For instance, the label $\max\,h_f$ means that we take as $h_\nu$
the maximum ($\max$) diameter among all the faces ($h_f$) sharing $\nu$.

We take into account only the set of \textbf{CVT} and \textbf{Random} meshes,
because in a structured mesh these choices of $h_\nu$ are really close to each other.

In Figures~\ref{fig:ese2Voro} and~\ref{fig:ese2Rand} we show the convergence lines
for the set of \textbf{CVT} and~\textbf{Random} meshes, respectively.

The trend of the $H^2$-seminorm error is similar for each choice of the scaling parameter $h_\nu$,
indeed the convergence lines are all indistinguishable except for  the minimum $h_e$ choice, 
(exhibiting the worst behavior),
see Figures~\ref{fig:ese2Voro} and~\ref{fig:ese2Rand} left.

The behavior of the $L^2$-norm error is more sensitive with respect to the parameter $h_\nu$ but 
it preserves a similar slope of the error in all cases.
Also in this case the minimum $h_e$ presents a worse behavior with respect to all the other ones.

\begin{figure}[!htb]
\begin{center}
\begin{tabular}{cc}
\includegraphics[width=0.45\textwidth]{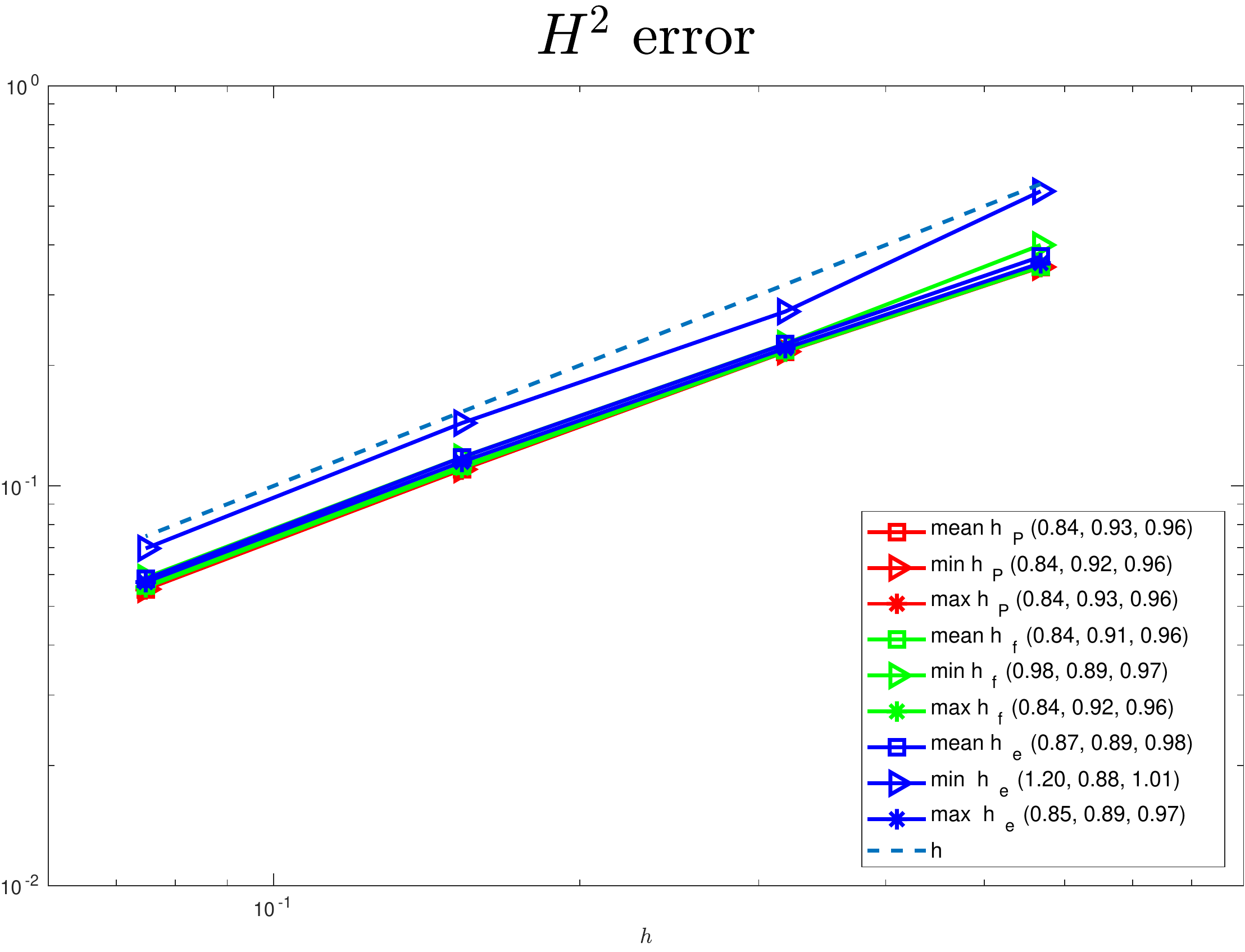} & 
\includegraphics[width=0.45\textwidth]{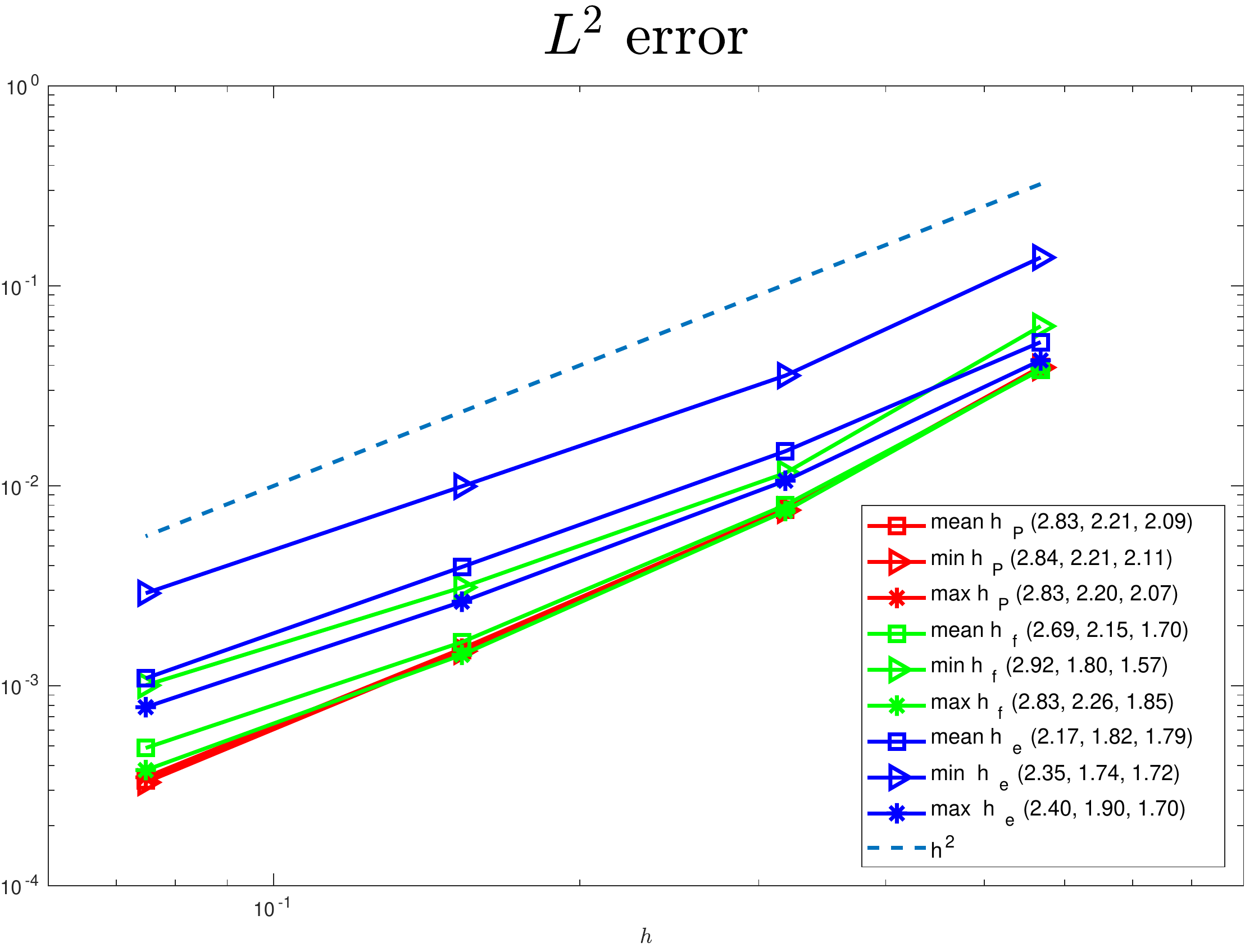}   \\
\end{tabular}
\end{center}
\caption{Example 2: trend of the $H^2$-seminorm (left) and $L^2$-norm (right) error with different scaling parameter $h_\nu$ for \textbf{CVT} meshes.
In the legend we report the \coor{convergence order} at each step.}
\label{fig:ese2Voro}
\end{figure}

\begin{figure}[!htb]
\begin{center}
\begin{tabular}{cc}
\includegraphics[width=0.45\textwidth]{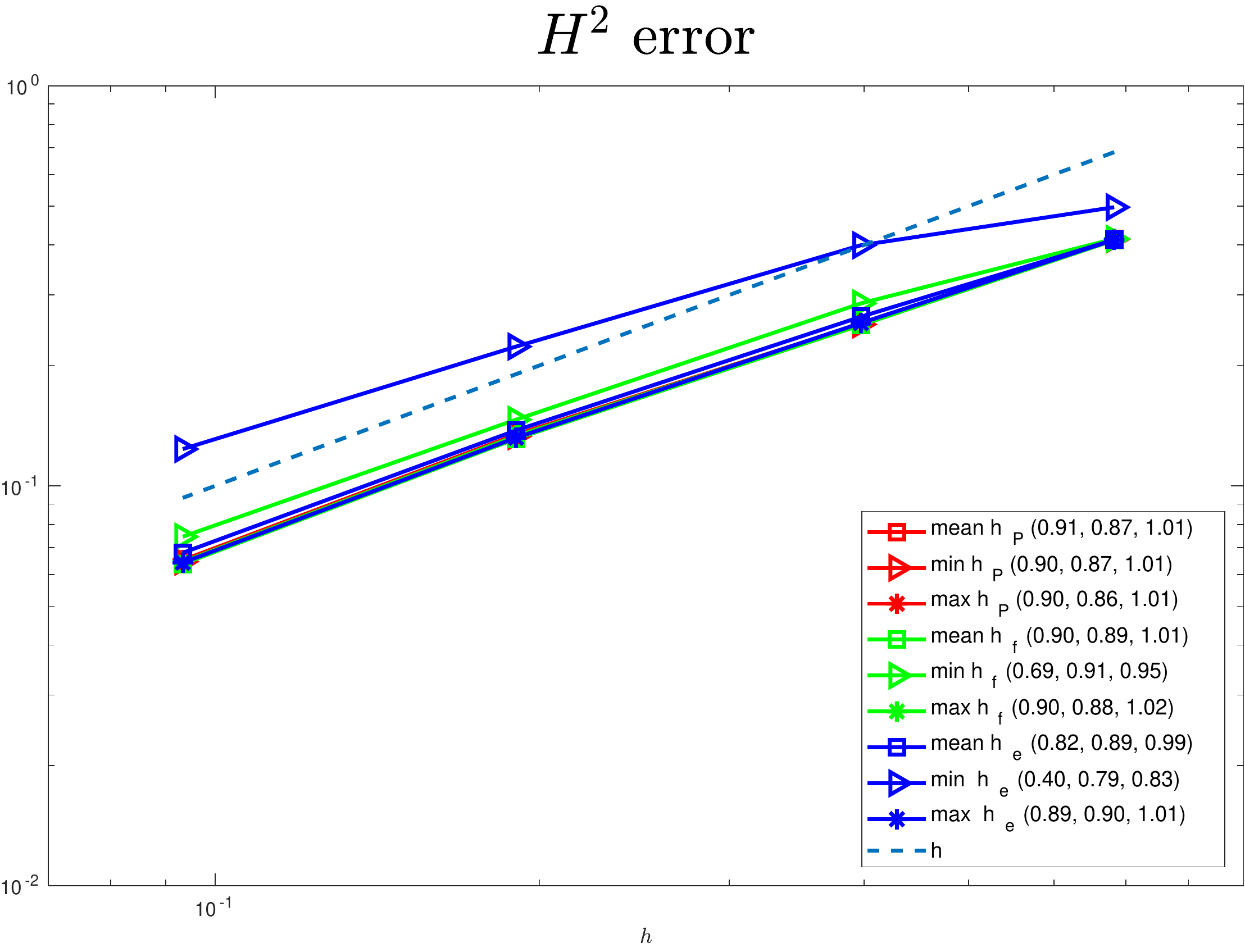} & 
\includegraphics[width=0.45\textwidth]{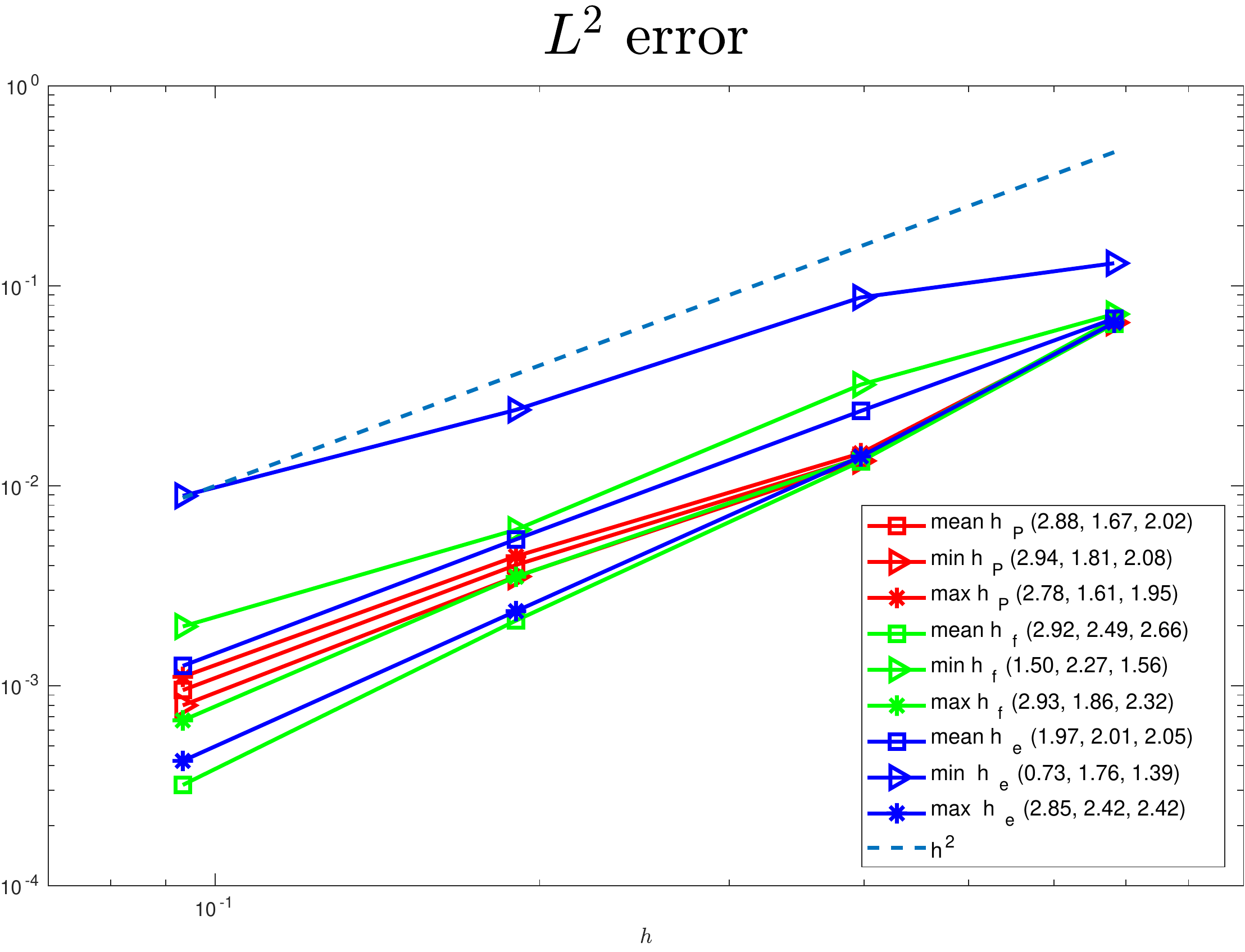}   \\
\end{tabular}
\end{center}
\caption{Example 2: trend of the $H^2$-seminorm (left) and $L^2$-norm (right) error with different scaling parameter $h_\nu$ for \textbf{Random} meshes.
In the legend we report the \coor{convergence order} at each step.}
\label{fig:ese2Rand}
\end{figure}

\subsection{Example 3: Comparison with $C^0$ VEM approach}\label{sub:c0Comp}

In this section we consider the same test of \coor{Section} 3.3
in~\cite{apollo11}.
We take the truncated octahedron as domain and
we solve the following second order partial differential equation
\begin{equation}
\left\{
\begin{array}{rlll}
-\Delta u + u &=&\, f\quad\quad&\textnormal{in }\Omega\\
u &=&\,g_1 \quad\quad&\textnormal{on }\partial\Omega\\
\end{array}
\right.,
\label{eqn:reacDiffC1C0}
\end{equation}
and we choose the right hand side $f$ and $g_1$ in such a way that the exact solution is 
$$
u(x,\,y,\,z):=\sin(2xy)\,\cos(z)\,.
$$

We consider three ``different'' $C^0$ VEM schemes  of degree 2 in 3D which differ in the number of degrees of freedom.
More specifically, given a mesh composed by $n_P$ polyhedrons, $n_f$ faces, $n_e$ edges and $n_\nu$ vertices,
we consider the following choices of $C^0$ or $C^1$ VEM approaches
\begin{itemize}
 \item \textbf{c1}: the $C^1$ method proposed in this paper,
 $$
 \#dofs=4 n_\nu\,,
 $$
 \item \textbf{c0}: a standard $C^0$ VEM~\cite{apollo11},
 $$
 \#dofs=n_\nu+n_e+n_f+n_P\,,
 $$
 \item \textbf{c0 static cond}: a standard $C^0$ VEM with static condensation of the internal-to-element degrees of freedom,
 $$
 \#dofs=n_\nu+n_e+n_f\,,
 $$
 \item \textbf{c0 sere}: a serendipity $C^0$ VEM with static condensation~\cite{chinaSere},
 $$
 \#dofs=n_\nu+n_e\,.
 $$
\end{itemize}
The mesh-size parameter behaves as $h\sim \# dofs^{-1/3}$
so the theoretical slope in a $\# dof$ vs $H^1$-seminorm error graph is expected to be $-2/3$ in all cases 
(that corresponds to $O(h^2)$ convergence rate).

In Figure~\ref{fig:ese3H1} we show the convergence lines of the $H^1$-seminorm error
for the set of meshes $\textbf{CVT}$ and $\textbf{Random}$ in terms of number of degrees of freedom.
From these convergence lines, we numerically show that all these methods have the expected error trend.
Moreover, we observe that the error values at each step are close to each other for all methods.

We also observe that 
for a given level of accuracy, 
the number of degrees of freedom to get a $C^1$ 
solution is approximately the same as for a $C^0$ scheme;
clearly, the number of degrees of freedom is only a rough 
indicator of the overall computational cost.

\begin{figure}[!htb]
\begin{center}
\begin{tabular}{cc}
\includegraphics[width=0.45\textwidth]{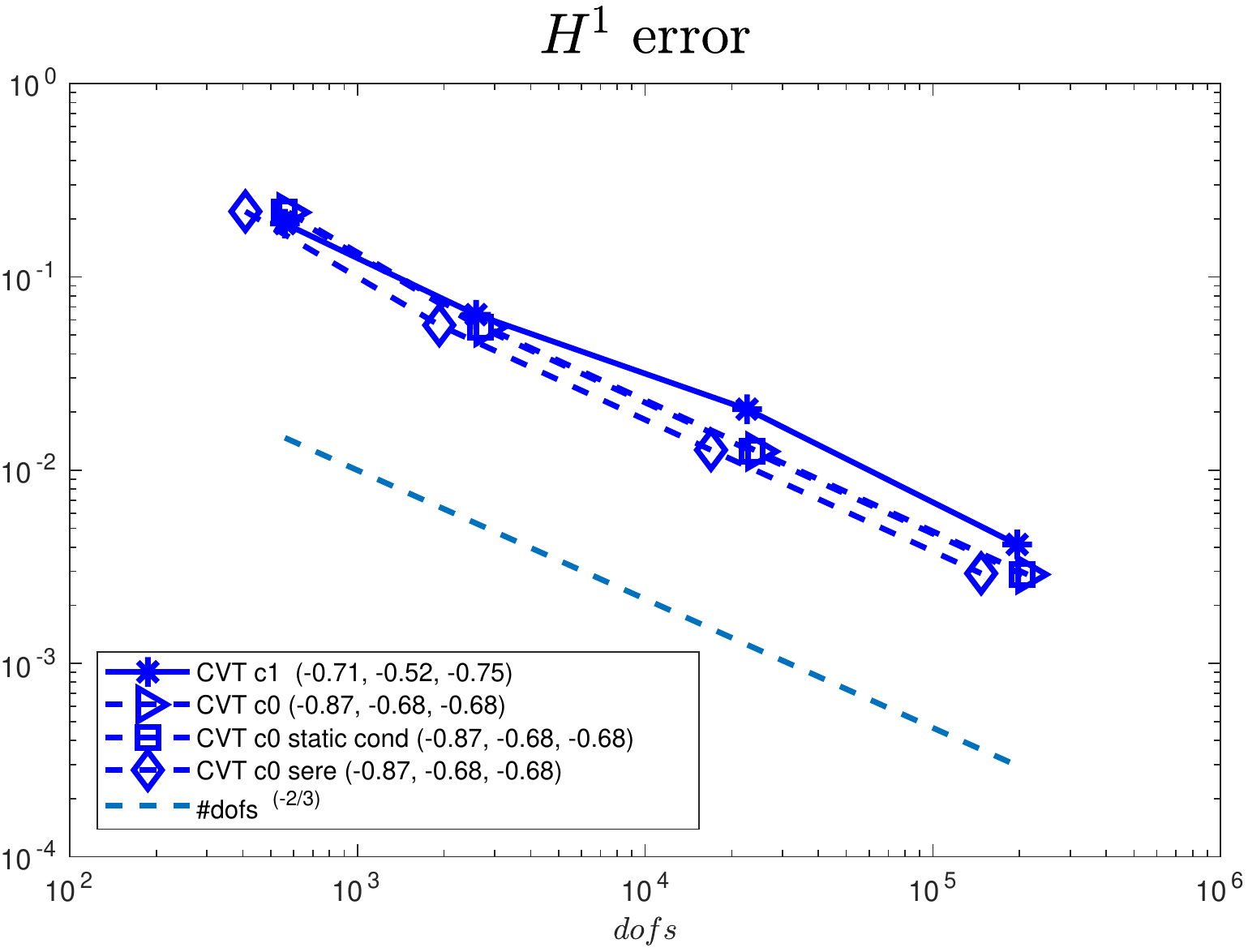} & 
\includegraphics[width=0.45\textwidth]{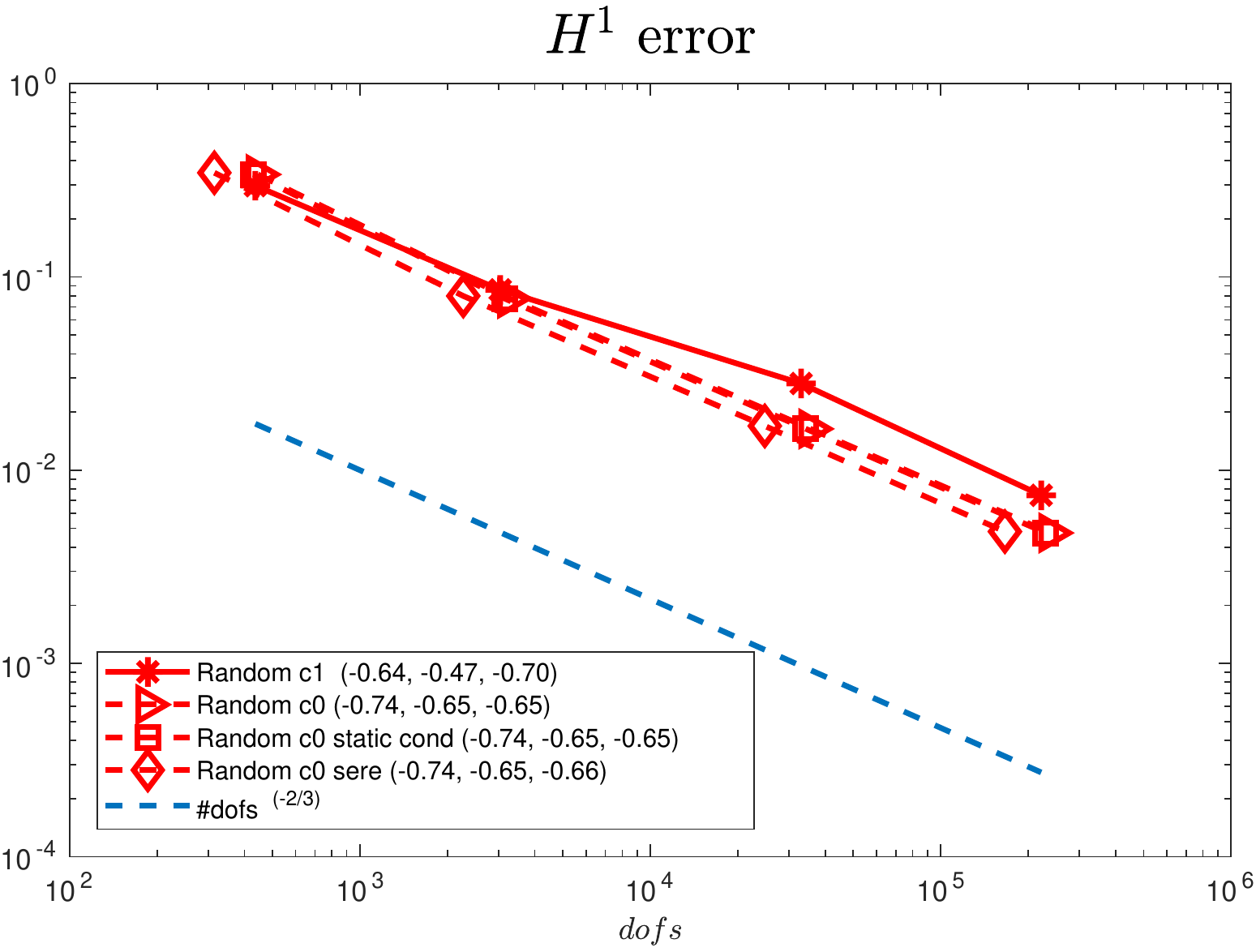}   \\
\end{tabular}
\end{center}
\caption{Example 3: trend of the $H^1$-seminorm error with \textbf{CVT} (left) and \textbf{Random} (right) meshes.
In the legend we report the \coor{convergence order} at each step.}
\label{fig:ese3H1}
\end{figure}

\coor{As a final comparison, we compare the running times (in seconds) for the $C^1$ scheme and the $C^0$ schemes (original and Serendipity version). The considered problem is~\eqref{eqn:reacDiffC1C0} and the adopted meshes are the CVT and random Voronoi. The results are reported in Table~\ref{tab:exe3Time}, where we present both the the assembling and solving time.
The code is runned in serial on a Linux machine with processor 
Intel(R) Xeon(R) CPU E5-2637 v4 @ 3.50GHz. The solution of the linear system is obtained via the direct solver provided by Pardiso~\cite{pardiso}.
We underline that the involved times could be reduced by running the whole code (both assembling and solving) in parallel, but this is beyond the scope of the present work.

From such data we observe that, as expected, both the assembling and solving time of $\textbf{c1}$ is sensibly greater than the other types. Indeed the $C^1$ scheme involves a more complex structure in terms of degrees of freedom and projection operators, in addition to having slightly more DoFs. On the other hand, this is only a rough time comparison based on our current C++ code; optimizing the codes could lead to smaller differences.
We also remind that the purpose of the $C^1$ scheme is on fourth order problems, and here we are only checking its performance on second order problems (for which it seems a viable, but not optimal, choice). 


\begin{table}[!htb]
\coor{
\begin{center}
\begin{tabular}{|c|rr|rr|rr|}
\cline{2-7}
\multicolumn{1}{c}{} &
\multicolumn{2}{|c|}{\textbf{c1}} &
\multicolumn{2}{ c|}{\textbf{c0}} & 
\multicolumn{2}{ c|}{\textbf{c0 sere}}\\
\cline{2-7}
\multicolumn{1}{c|}{} &
$T_{ass}$ &$T_{sol}$ &$T_{ass}$ &$T_{sol}$  &$T_{ass}$ &$T_{sol}$ \\
\hline
\textbf{CVT} 1 &1    &$\approx$ 0  &$\approx$ 0 &$\approx$ 0  &$\approx$ 0 &$\approx$ 0 \\
\textbf{CVT} 2 &5    &$\approx$ 0  &1           &$\approx$ 0  &$\approx$ 0 &$\approx$ 0  \\
\textbf{CVT} 3 &8    &160          &5           &3            &5           &2  \\
\textbf{CVT} 4 &1695 &854          &49          &208          &44          &98 \\
\hline
\end{tabular}
\end{center}
\begin{center}
\begin{tabular}{|c|rr|rr|rr|}
\cline{2-7}
\multicolumn{1}{ c|}{} &
\multicolumn{2}{ c|}{\textbf{c1}} &
\multicolumn{2}{ c|}{\textbf{c0}} & 
\multicolumn{2}{ c|}{\textbf{c0 sere}}\\
\cline{2-7}
\multicolumn{1}{c|}{} &
$T_{ass}$ &$T_{sol}$ &$T_{ass}$ &$T_{sol}$ &$T_{ass}$ &$T_{sol}$ \\
\hline
\textbf{Random} 1 &$\approx$ 0 &1               &$\approx$ 0 &$\approx$ 0 &$\approx$ 0 &$\approx$ 0  \\
\textbf{Random} 2 &7           &1               &1           &$\approx$ 0 &$\approx$ 0 &$\approx$ 0  \\
\textbf{Random} 3 &74          &11              &8           &8           &7           &4            \\
\textbf{Random} 4 &884         &641             &59          &325         &52          &145          \\
\hline
\end{tabular}
\end{center}
\caption{Example 3, time comparison among $C^1$ and $C^0$ VEM. Assembling time in seconds, $T_{ass}$, and solving time in seconds, $T_{sol}$, for different meshes of the CVT and Random Voronoi families.}
\label{tab:exe3Time}}
\end{table}}

\section*{Acknowledgments}

The first and second authors were partially supported by the European Research Council through the H2020 Consolidator Grant (grant no. 681162) CAVE~--~Challenges and Advancements in Virtual Elements. This support is gratefully acknowledged. 


\appendix
\section{A glimpse to the general order case}\label{appendix:general}

In this appendix we give an hint on the general order case $k>2$,
the lowest order case presented in the paper corresponds to $k=2$.
We limit ourself to the definition of the local face and polyhedral spaces, i.e.
we consider the same work-flow of Section~\ref{sec:spaces} but
omit the proof of these results which can be deduced from the lowest order case.

Starting from the functional spaces defined in the following paragraphs,
the generalization of the discrete forms defined in Equation~\eqref{eqn:discLocalForm} becomes 
technical but straightforward.

To define such functional spaces, 
we define the polynomial space $\mathbb{P}_s\backslash\mathbb{P}_r(\mathcal{D})$ 
as any complement space of $\mathbb{P}_r(\mathcal{D})$, i.e. that verifies 
$$
\mathbb{P}_s(\mathcal{D}) = (\mathbb{P}_s\backslash\mathbb{P}_r(\mathcal{D}))\oplus\,\mathbb{P}_r(\mathcal{D})\,,
$$
where $\mathcal{D}$ is a generic two or three dimensional domain and the integers $s>r\geq 0$.

\vspace{1em}
\noindent\textbf{Virtual element nodal space $\VFNGen$}
\vspace{1em}

\noindent We start from the preliminary space
\begin{eqnarray*}
\VFNTGen := \bigg\{v_h\in H^1(f)\:&:&\Delta_\tau\,v_h\in\mathbb{P}_{k-2}(f)\,, \\[-1em]
\phantom{\VFNTGen := \bigg\{v_h\in H^1(f)\:}&\phantom{:}&
v_h|_{\partial f}\in C^0(\partial f)\,,\:v_h|_e\in\mathbb{P}_{k-1}(e)\:\:\forall e\in\partial f\bigg\}\,. 
\end{eqnarray*}

We build the projection operator $\Pi^\nabla_{f,k}:\VFNTGen\to\mathbb{P}_{k-1}(f)$,
defined in a similar way as in Equation~\eqref{eqn:piNablaFace}, and the functional space

\begin{equation*}
\VFNGen := \left\{ v_h\in\VFNGen\::\:\int_f\Pi^\nabla_{f,k} v_h\,q\df = \int_f v_h\,q\df,
\quad \forall q\in\mathbb{P}_{k-2}\backslash\mathbb{P}_{k-3}(f)\right\}\,.
\label{eqn:nodalEnhancedGen}
\end{equation*}

The degrees of freedom of such space are the standard nodal VEM ones and 
they are enough to define $\Pi^\nabla_{f,k}$.
This virtual element space is standard in the virtual element community,
we refer to~\cite{volley} for more details.

\vspace{1em}
\noindent\textbf{Virtual element $C^1$ space $\VFGen$}
\vspace{1em}

\noindent We generalize the preliminary space defined in Subsection~\ref{sub:VSOnFace2} as
\begin{eqnarray*}
\VFTGen := \Bigg\{v_h\in H^2(f)\:&:&\: \Delta^2_\tau\,v_h\in\mathbb{P}_{k-1}(f),\\[-1em]
\:&\phantom{:}&\: v_h|_{\partial f}\in C^0(\partial f),\:\:v_h|_e\in\mathbb{P}_k(e)\:\:\forall e\in\partial f\,, \\
\:&\phantom{:}&\: \nabla_\tau\,v_h|_{\partial f}\in [C^0(\partial f)]^2,\:\:\\[-1em]
\:&\phantom{:}&\: \partial_{n_e} v_h\in\mathbb{P}_{k-1}(e)\:\:\forall e\in\partial f\Bigg\}\,,
\end{eqnarray*}  

Starting from the projection operator $\Pi^\Delta_{f,k}:\VFTGen\to\mathbb{P}_k(f)$ defined in a similar way as
$\Pi^\Delta_f$, see Equation~\eqref{eqn:piDeltaFace},
we are able to define the virtual space
\begin{equation*}
\VFGen:= \left\{ v_h\in\VFTGen\::\:\int_f \Pi^\Delta_{f,k} v_h\,\,q\df = \int_f v_h\,\,q\df,
\quad\forall q\in\mathbb{P}_{k-1}\backslash\mathbb{P}_{k-4}(f)\right\}\,,
\label{eqn:c1EnhancedGen} 
\end{equation*}

Also in this case the degrees of freedom are enough to define $\Pi^\Delta_{f,k}$.
This space face is similar to the one defined in~\cite{BM13} and we refer to this paper for more details.

Moreover, it is easy to show that starting from the degrees of freedom of $\VFGen$,
it is possible to define the $L^2$ projection operator to approximate the gradient of a function $v_h\in \VFGen$, i.e.
$\bPi^0_{f,k-1}$ (the counterpart of the operator defined in Equation~\eqref{eqn:defPiFGrad}).

\vspace{1em}
\noindent\textbf{General order virtual element space in $P$}
\vspace{1em}

\noindent Given a polyhedron $P\in\Omega_h$ we consider the preliminary space 
\begin{eqnarray*}
\widetilde{V}_{h,k}(P) := \bigg\{v_h\in H^2(P)\:&:&\: \Delta^2\,v_h\in\mathbb{P}_k(P), \nonumber \\[-0.3em]
\:&\phantom{:}&\: 
v_h|_{S_P}\in C^0(S_P)\,,
\nabla v_h|_{S_P}\in[C^0(S_P)]^3\,,\nonumber \\
\:&\phantom{:}&\: 
v_h|_f\in \VFGen\,,\,
\partial_{n_f} v_h|_f\in \VFNGen
\quad\forall f\in\partial P
\bigg\}\,,\nonumber \\
\end{eqnarray*}

In this virtual element space we define the following set of linear operators from $\widetilde{V}_{h,k}(P)$ to $\mathbb{R}$:
\begin{center}
\begin{minipage}{1.\textwidth}
\begin{itemize}
 \item[\DZ:] the values of the function at the vertices, $v_h(\nu)$; 
 \item[\DU:] the values of the gradient components at the vertices, $\nabla v_h(\nu)$;
 \item[\textbf{D2}:] values of $v_h(\nu)$ at $\max\{k-3,0\}$  internal points for each edge of~$\partial P$;
 \item[\textbf{D3}:] values of $\nabla v_h(\nu)$ along two normal-to-edge directions at $k-2$ internal points for each edge of~$\partial P$;
 \item[\textbf{D4}:] for each face $f\in\partial P$ we define the function moments 
 $$
 \int_f v_h\,q\df\qquad\forall q\in\mathbb{P}_{k-4}(f)\,,
 $$
 \item[\textbf{D5}:] for each face $f\in\partial P$ we define the gradient moments 
 $$
 \int_f (\nabla v_h\cdot\N_f)\,q\df\qquad\forall q\in\mathbb{P}_{k-3}(f)\,,
 $$
 \item[\textbf{D6}:] the internal function moments 
 $$
 \int_P v_h\,q\df\qquad\forall q\in\mathbb{P}_{k-4}(P)\,.
 $$
\end{itemize} 
\end{minipage}
\end{center}

Starting from these linear operators it is possible to determine the projection operator  
$\Pi^\Delta_{P,k}:\widetilde{V}_h(P)\to\mathbb{P}_k(P)$ defined in a similar way as in Equation~\eqref{eqn:piDelta}.
It is easy to see that this projection operator is uniquely defined by $\textbf{D1}-\textbf{D6}$ through
the face projectors $\Pi^\nabla_{f,k}$, $\Pi^\Delta_{f,k}$ and $\bPi^0_{f,k-1}$.
Then the $C^1$ general order virtual elements space defined on polyhedron reads
\begin{eqnarray*}
 V_h(P) &:=& \Bigg\{v_h\in \widetilde{V}_{h,k}(P)\::\: 
\mathlarger{\int_P} \Pi^\Delta_{P,k} v_h\,\,q\dPP = \mathlarger{\int_P} v_h\,\,q\dPP\,,
\quad\forall q\in\mathbb{P}_{k}\backslash\mathbb{P}_{k-4}(P)
\Bigg\}\,.
\end{eqnarray*}

%

%
\bibliographystyle{plain}
\bibliography{VEM}

\end{document}